\documentclass[12pt]{amsart}
\usepackage{amsmath,amssymb,amscd,amsthm,amsbsy,times}
\usepackage[initials,short-journals]{amsrefs}
\DeclareFontFamily{U}{rsfs}{\skewchar\font"7F}
\DeclareFontShape{U}{rsfs}{m}{n}{
	<-6> rsfs5
	<6-8> rsfs7
	<8-> rsfs10
	}{}
\DeclareMathAlphabet{\mathscr}{U}{rsfs}{m}{n}
\usepackage[shortlabels,inline]{enumitem}
\usepackage[hidelinks]{hyperref}
\usepackage[all]{xy}
\newcommand{\CF}{\mathcal{F}}
\newcommand{\Hom}{\mathrm{Hom}}
\newcommand{\gl}{\mathfrak{gl}}
\newcommand{\GL}{\mathrm{GL}}
\newcommand{\PGL}{\mathrm{PGL}}
\newcommand{\tgl}{\mathfrak{tgl}}
\newcommand{\SL}{\mathrm{SL}}

\newcommand{\id}{{\mathrm{id}}}
\DeclareMathOperator{\Ad}{\mathrm{Ad}}
\DeclareMathOperator{\ad}{\mathrm{ad}}

\newcommand{\pdif}[2]{\frac{\partial#1}{\partial#2}}

 \newtheorem{theorem}{Theorem}[section]
 
 \newtheorem{proposition}[theorem]{Proposition}
 \newtheorem{lemma}[theorem]{Lemma}

\theoremstyle{definition}
 \newtheorem{definition}[theorem]{Definition}
 \newtheorem{notation}[theorem]{Notation}
 \newtheorem{example}[theorem]{Example}
 \newtheorem{remark}[theorem]{Remark}
\title{Formal frames and deformations of affine connections}
\author{Taro Asuke}
\address{Graduate School of Mathematical Sciences, University of Tokyo, 3-8-1 Komaba, Meguro-ku, Tokyo 153-8914, Japan}
\email{asuke@ms.u-tokyo.ac.jp}
\date{June 23, 2022\\
\indent\textit{Revised}: March 14, 2023}
\thanks{This work is partly supported by JSPS KAKENHI Grant Number JP21H00980.}
\begin{document}
\begin{abstract}
We will introduce formal frames of manifolds, which are a generalization of ordinary frames.
Their fundamental properties are discussed.
In particular, canonical forms are introduced, and torsions are defined in terms of them as a generalization of the structural equations.
It will be shown that the vanishing of torsions are equivalent to the realizability of given formal frames as ordinary frames.
We will also discuss deformations of linear connections on tangent bundles.
An application to deformations of foliations are then given.
\end{abstract}
\maketitle
\setlength{\baselineskip}{16pt}
\section*{Introduction}
The theory of frames are well-developed.
It is based on the local diffeomorphisms of $\mathbb{R}^n$ which fix the origin, and their jets.
As a consequence, tensors, etc. are commutative in the lower indices.
For example, if $T$ is a related tensor and if $T^i{}_{jk}$ denote the components of $T$ with respect to a chart, then we have $T^i{}_{kj}=T^i{}_{jk}$.
The bundle of $2$-frames are quite related with connections.
The commutativity of the indices imply that these connections are torsion-free.
On the other hand, when we consider geometric structures, usually connections with torsions appear.
Garc\'\i a gave a framework which enables us to work on connections with torsions in~\cite{Garcia}.
These two frameworks are similar but differ at several points.
For example, the structural group in the Kobayashi construction~\cite{K_str} is much larger than Garc\'\i a's one and its action has a clear meaning.
On the other hand, Garc\'\i a's construction has an an advantage that any principal bundles can be treated, while the theory of frames work basically on manifolds and their frame bundles.
In this paper, we introduce a notion of formal frames, which is a kind of frames with non-commutative indices, and by which we can understand the both frameworks if we restrict ourselves to frame bundles.
We will introduce canonical forms like in the classical case (Definitions~\ref{defG1}, \ref{def4.4}).
Such canonical forms are studied by Garc\'\i a~\cite{Garcia} in basic cases.
Canonical forms in this article are generalizations of these classical ones.
Then, we will define torsions in terms of canonical forms (Definition~\ref{def5.2}).
They are a generalization of the structural equations in the classical setting.
It will be shown that a formal frame is actually a classical frame if and only if every torsion vanishes (Theorem~\ref{thm5.7}).
Integrability of formal frames under geometric structure can be also formulated on the bundle of formal frames in some cases (Remark~\ref{rem5.9}, Example~\ref{ex5.11}).
We then discuss deformations of linear connections on tangent bundles.
We will show that infinitesimal deformations of connections can be regarded as connections on certain principal bundles (Theorem~\ref{thm6.16}).
Finally, we will discuss deformations of foliations as an application.\par
Throughout this article, we will make use of the Einstein convention.
That is, a pair of upper and lower indices of the same letter is understood to run from $1$ to $n$ ($q$ in Section~\ref{sec6}) and be taken the sum.
For example, $f^i\pdif{}{x^i}$ means $\sum_{i=1}^nf^i\pdif{}{x^i}$.
When we compare representations of objects with respect to two charts, we represent one in plain letters and another one by adding `$\widehat{\;\ \;}$'.
For example, let $T$ be a tensor of type $(1,2)$.
If $(U,\varphi)$ and $(\widehat{U},\widehat{\varphi})$ are charts, then the components of $T$ with respect to $(U,\varphi)$ are represented by $T^i{}_{jk}$, and the ones with respect to $(\widehat{U},\widehat{\varphi})$ are represented by $\widehat{T}^i{}_{jk}$.
When we deal with a Lie group, its Lie algebra is represented by corresponding German letter, e.g.~if $G$ is a Lie group, then its Lie algebra is represented by $\mathfrak{g}$.
Finally, we always make use of the standard coordinates for $\mathbb{R}^n$.

\section{The bundle of formal $2$-frames}

\begin{notation}
Let $M$ be a manifold and $p\in M$.
If $f$ is a mapping defined on a neighborhood of $p$, then we say that $f$ is a mapping from $(M,p)$.
Precisely speaking, we consider the germ of $f$ at $p$.
If the target of $f$, say $N$, is specified, then we say that $f$ is a mapping from $(M,p)$ to $N$.
If moreover the image $f(p)$ is specified, then we say that $f$ is a mapping from $(M,p)$ to $(N,f(p))$.
If $f$ is a diffeomorphism to the image, we say that $f$ is a \textit{local diffeomorphism}.
\end{notation}

\begin{notation}
Let $V\to M$ and $W\to N$ be vector bundles.
If $F$ is a bundle morphism from $(V,M)$ to $(W,N)$, then the underlying map from $M$ to $N$ is represented by $f$.
The mapping induced on fibers are represented as $F_p\colon V_p\to W_{f(p)}$, where $V_p$ denotes the fiber of $V$ over $p\in M$.
When we consider bundle automorphisms, we do \textit{not} assume the underlying maps to be the identity.
\end{notation}

Hence a bundle morphism $F$ consists of a pair $(f,F_{\bullet})$.

\begin{definition}
\begin{enumerate}
\item
If $h=(h^i)$ is an $\mathbb{R}^m$-valued function on an open subset of $\mathbb{R}^n$, then we define an $M_{m,n}(\mathbb{R})$-valued function $Dh$ by setting $Dh^i{}_j=\pdif{h^i}{x^j}$, where $(x^1,\ldots,x^n)$ are the standard coordinates for $\mathbb{R}^n$.
\item
If $h$ is an $M_{m,n}(\mathbb{R})$-valued function, then we represent $h$ as $h^i{}_j$ and set $Dh^i{}_{jk}=\pdif{h^i{}_j}{x^k}$.
\end{enumerate}
\end{definition}

We recall the notion of frames of higher order~\cite{K_str}.

\begin{definition}
Let $r\geq1$.
For $p\in M$, an \textit{$r$-frame} at $p$ is the $r$-jet at $o\in\mathbb{R}^n$ of a local diffeomorphism, say $f$,  from $(\mathbb{R}^n,o)$ to $(M,p)$ and is represented by $j^r_o(f)$.
The set of $r$-frames are represented by $P^r(M)$.
We set $\pi^r(j^r_o(f))=f(o)$.
If the target of $o$ is assumed to be a fixed point $p$, then we represent $P^r(M)$ as $P^r(M,p)$.
\end{definition}

The following is classical~\cite{K}.

\begin{theorem}
If we set $G^r_n=P^r(\mathbb{R}^n,o)$, then $G^r_n$ is a Lie group of dimension $n\sum_{i=1}^r\dbinom{n}{i}$ and $P^r(M)$ is naturally a principal $G^r_n$-bundle.
Indeed, the action is given by compositions.
\end{theorem}

\begin{remark}
The arguments and calculations in Sections~1, 2 are similar to those of~\cite{Kobayashi-Nagano}.
\end{remark}

Let now $\pi^1\colon P^1(M)\to M$ be the $1$-frame bundle, which is naturally isomorphic to the structure bundle of $TM$.
We represent $P^1(M)$ as $P^1$ if $M$ is clear.
Let $F=(f,F_\bullet)\colon T(\mathbb{R}^n,o)\to TM$ be a bundle isomorphism to the image such that $F_o=Df(o)$.
If $(U,\varphi)$ is a chart about $f(o)$, then $D\varphi\circ F$ is represented as $(D\varphi\circ F)_x=A^i{}_j(x)\pdif{}{y^i}_{f(x)}$, where $x\in\mathbb{R}^n$ is in a neighbourhood of $o$ and $(y^1,\dots,y^n)$ are coordinates for $\varphi(U)$.
Note that if we represent $\varphi\circ f$ as $\varphi\circ f(x)=\varphi(f(o))+a^i{}_jx^j+\overline{f}(x)$, where $\overline{f}$ is of order greater than one with respect to $x$, then we have $a^i{}_j=A^i{}_j(o)$.
Hence $j^1_o(F)$ can be represented by a triple $(a^i,a^i{}_j,a^i{}_{jk})\in\mathbb{R}^n\times\GL_n(\mathbb{R})\times\mathbb{R}^{n^3}$, where $a^i=(\varphi(f(o))^i$ and $a^i{}_{jk}=\pdif{}{x^k}A^i{}_j(o)$.

\begin{definition}
We set
\[
\widetilde{G}_n^2=\{j^1_o(F)\mid\text{$F$ is a bundle automorphism of $T(\mathbb{R}^n,o)$ such that $F_o=Df(o)$}\}.
\]
\end{definition}

Let $F,G\colon T(\mathbb{R}^n,o)\to T(\mathbb{R}^n,o)$.
The underlying map of $F\circ G$ is equal to $f\circ g$, and we have $(F\circ G)_x=F_{g(x)}\circ G_x$.
In particular, we have $f\circ g(o)=f(o)=o$ and that $(F\circ G)_o=F_{g(o)}\circ G_o=F_o\circ G_o$.
Hence $\widetilde{G}_n^2$ admits a group structure of which the product is the composition.
Indeed, $\widetilde{G}_n^2$ is a Lie group diffeomorphic to $\GL_n(\mathbb{R})\times\mathbb{R}^{n^3}$.
We will describe the product in Lemma~\ref{lem1.10}.

\begin{definition}
\label{def1.6}
Let
\[
\widetilde{P}^2(M)=\left\{j^1_o(F)\;\middle|\;\parbox[c]{18.4em}{$F\colon T(\mathbb{R}^n,o)\to TM$ is a bundle isomorphism\\ to the image such that $F_o=Df(o)$}\right\}.
\]
We call $\widetilde{P}^2(M)$ the \textit{bundle of formal $2$-frames} and also the formal $2$-frame bundle for short.
An element of $\widetilde{P}^2(M)$ is called a \textit{formal frame} of order $2$.
If $u=j^1_o(F)\in\widetilde{P}^2(M)$, then we set $\pi^2(u)=f(o)$.
We regard $F(o)$ as an element of $P^1(M)$ and set $\pi^2_1(u)=F(o)$.
If in addition $a=j^1_o(G)\in\widetilde{G}_n^2$, then we set $u.a=j^1_o(F\circ G)$.
\end{definition}
Note that we have $\pi^2=\pi^1\circ\pi^2_1$ and that $\widetilde{G}^2_n=\widetilde{P}^2(\mathbb{R}^n,o)$.

Let $p\in M$ and $(U,\varphi)$ be a chart about $p$.
If $u\in(\pi^2)^{-1}(U)$, then we represent $u=j^1_o(F)$, where $F$ is a bundle isomorphism.
We set $f_\varphi=\varphi\circ f$, $F_\varphi=D\varphi\circ F$ and associate with $u$ a triple $\left(f_\varphi(o)^i,F_\varphi(o)^i{}_j,\pdif{F_\varphi{}^i{}_j}{x^k}(o)\right)$, where $(x^1,\ldots,x^n)$ are the standard coordinates for $\mathbb{R}^n$.
We do not distinguish $F_\varphi(o)^i$ and $F_\varphi{}^i(o)$ and so on in what follows.

\begin{notation}
\label{not1.6}
We refer to the coordinates for $\widetilde{P}^2$ as above as the \textit{natural coordinates} for $\widetilde{P}^2$ associated with $\varphi$ after \cite{K}*{p.~140}.
\end{notation}

If $M=\mathbb{R}^n$, $\varphi=\id$ and if $u\in\widetilde{G}_n^2$, then $F_\varphi(o)^i=0$ so that we can associate with $u$ a pair $\left(F_\varphi(o)^i{}_j,\pdif{F_\varphi{}^i{}_j}{x^k}(o)\right)$.
Let $a=j^1_o(G)\in\widetilde{G}_n^2$, where $g(o)=o$.
Then $u.a=j^1_o(F\circ G)$ is represented with respect to $\varphi$ as
\stepcounter{theorem}
\begin{align*}
\tag{\thetheorem}
\label{eqF4}
&\hphantom{{}={}}%
\left(F_\varphi(o)^i,F_\varphi(g(o))^i{}_\alpha G(o)^\alpha{}_j,\pdif{F_\varphi{}^i{}_\alpha}{x^\beta}(g(o))G(o)^\alpha{}_jDg(o)^\beta{}_k+F_\varphi(o)^i{}_\alpha\pdif{G^\alpha{}_j}{x^k}(o)\right)\\*
&=\left(F_\varphi(o)^i,F_\varphi(o)^i{}_\alpha(o)G(o)^\alpha{}_j,\pdif{F_\varphi{}^i{}_\alpha}{x^\beta}(o)G(o)^\alpha{}_jG(o)^\beta{}_k+F_\varphi(o)^i{}_\alpha\pdif{G^\alpha{}_j}{x^k}(o)\right).
\end{align*}
where $Dg(o)^i{}_j=G(o)^i{}_j$ because $j^1_o(G)\in\widetilde{G}_n^2$.

Let $(U,\varphi)$ and $(U,\widehat{\varphi})$ be charts and $\phi=\widehat{\varphi}\circ\varphi^{-1}$ be the transition function.
Let $u=j^1_o(F)\in(\pi^2)^{-1}(U)\subset\widetilde{P}^2$ and represent $u$ as $(h^i,h^i{}_j,h^i{}_{jk})$ in the natural coordinates associated with $(U,\varphi)$.
If we represent $u$ with respect to $(U,\widehat{\varphi})$, then $F_{\widehat{\varphi}}=(D\phi\circ f)\circ F_{\varphi}$ so that we have
\begin{align*}
(\widehat{h}^i,\widehat{h}^i{}_j,\widehat{h}^i{}_{jk})&=(\phi(h^a)^i,D\phi(h^a)^i{}_\alpha h^\alpha{}_j,H\phi(h^a)^i{}_{\alpha\beta}h^\alpha{}_jh^\beta{}_k+D\phi(h^a)^i{}_\alpha h^\alpha{}_{jk})\\*
&=(\phi(h^a)^i,(D\phi(h^a)^i{}_j,H\phi(h^a)^i{}_{jk})(h^i{}_j,h^i{}_{jk})),
\end{align*}
where the product in the most right hand side is taken in $\widetilde{G}_n^2$.

Summing up, we have the following

\begin{lemma}
\label{lem1.10}
\begin{enumerate}
\item
Let $a,b\in\widetilde{G}_n^2$ and represent $a,b$ as $a=(a^i{}_j,a^i{}_{jk})$, $b=(b^i{}_j,b^i{}_{jk})$.
Then, we have
\[
ab=(a^i{}_\alpha b^\alpha{}_j,a^i{}_{\alpha\beta}b^\alpha{}_jb^\beta{}_k+a^i{}_\alpha b^\alpha{}_{jk}).
\]
\item
The bundle of formal $2$-frames $\widetilde{P}^2$ is a principal $\widetilde{G}_n^2$-bundle with the projection $\pi^2$.
\end{enumerate}
\end{lemma}

The bundles $\widetilde{P}^2$ and $P^2$, the groups $\widetilde{G}_n^2$ and $G_n^2$ are related as follows.
First we introduce the following
\begin{definition}
Let $u\in P^2$ and represent $u$ as $u=j^2_o(f)$, where $f\colon(\mathbb{R}^n,o)\to M$ is a local diffeomorphism.
We set $\epsilon(u)=j^1_o(Du)$.
\end{definition}

The map $\epsilon$ gives an inclusion of $G_n^2$ to $\widetilde{G}_n^2$.

There is also a projection from $\widetilde{P}^2$ to $P^2$.
First, if $j^1_o(F)\in\widetilde{P}^2$, then we may modify $f$ as follows.
Let $(U,\varphi)$ be a chart and let $(y^1,\ldots,y^n)$ be the standard coordinates for $\varphi(U)$.
We have
\begin{align*}
(\varphi\circ f(x))^i&=\varphi(f(o))^i+a^i{}_\alpha x^\alpha+\overline{f}(x)^i,\\*
F_\varphi(x)&=\pdif{}{y^i}(a^i{}_j+a^i{}_{j\alpha}x^\alpha+\overline{a^i{}_j}(x)),
\end{align*}
where $\overline{f}$ and $\overline{a^i{}_j}$ are of order greater than $1$ with respect to $x^i$.
As we are concerned with $1$-jets, we set $b^i{}_{jk}=\frac12(a^i{}_{jk}+a^i{}_{kj})$ and modify $\overline{f}$ as $\frac12b^i{}_{jk}x^jx^k+\overline{f}'(x)^i$, where $\overline{f}'(x)$ is of order greater than $2$.
After this modification, we have $H(\varphi\circ f)(o)^i{}_{jk}=b^i{}_{jk}$.
This property is stable in the following sense.
First, let $(U,\widehat{\varphi})$ be also a chart and set $\phi=\widehat{\varphi}\circ\varphi^{-1}$.
As we have just seen, we have
\[
(\widehat{a}^i{}_j,\widehat{a}^i{}_{jk})=(D\phi(p)^i{}_j,H\phi(p)^i{}_{jk})(a^i{}_j,a^i{}_{jk}),
\]
where $p=\varphi(f(o))$.
Since $H\phi(p)^i{}_{ml}=H\phi(p)^i{}_{lm}$, we have
\begin{align*}
\widehat{b}^i{}_{jk}&=\frac12(\widehat{a}^i{}_{jk}+\widehat{a}^i{}_{kj})\\*
&=\frac12(H\phi(p)^i{}_{\alpha\beta}a^\alpha{}_ja^\beta{}_k+D\phi(p)^i{}_\alpha a^\alpha{}_{jk}+H\phi(p)^i{}_{\alpha\beta}a^\alpha{}_ka^\beta{}_j+D\phi(p)^i{}_\alpha a^\alpha{}_{kj})\\*
&=H\phi(p)^i{}_{\alpha\beta}a^\alpha{}_ja^\beta{}_k+D\phi(p)^i{}_\alpha\left(\frac12(a^\alpha{}_{jk}+a^\alpha{}_{kj})\right)\\*
&=H\phi(p)^i{}_{\alpha\beta}a^\alpha{}_ja^\beta{}_k+D\phi(p)^i{}_\alpha b^\alpha{}_{jk}.
\end{align*}
On the other hand, we have
\stepcounter{theorem}
\begin{align*}
\tag{\thetheorem}
\label{eqF8}
&\hphantom{{}={}}
\phi\circ(\varphi\circ f)(x)\\*
&=\widehat{\varphi}(f(o))+D\phi(p)^i{}_\alpha a^\alpha{}_\beta x^\beta\\*
&\hphantom{{}={}}+\frac12\left(H\phi(p)^i{}_{\alpha\beta}a^\alpha{}_ja^\beta{}_k+D\phi(p)^i{}_\alpha b^\alpha{}_{jk}\right)x^jx^k+(\text{terms of order greater than $2$}).
\end{align*}

This means that the $2$-jets of modified mappings are independent of the choice of charts.
Similarly, we see that the modification is compatible with products with elements of $\widetilde{G}^2_n$.
Hence the following definitions make sense.

\begin{definition}
Let $j^1_o(F)\in\widetilde{P}^2$.
If we have $F_o=Df(o)$ and that $Hf(o)^i{}_{jk}=\frac12(DF(o)^i{}_{jk}+DF(o)^i{}_{kj})$, then we say $F$ is a \textit{normal} representative.
\end{definition}

Note that if we begin with a local diffeomorphism $f\colon(\mathbb{R}^n,o)\to M$ and if we consider $j^1_o(Df)$, then $Df$ is normal as a representative.

\begin{definition}
Let $u\in\widetilde{P}^2$.
We choose a normal representative $F$ for $u$ and set $\kappa(u)=j^2_o(f)$, where $F=(f,F_\bullet)$.
\end{definition}

From these arguments, we see the following

\begin{theorem}
\label{thm1.13}
\begin{enumerate}
\item
The mapping $\epsilon\colon P^2\to\widetilde{P}^2$ is well-defined and is an embedding of\/ $G_n^2$-bundles.
Moreover, we have $\widetilde{P}^2=\epsilon(P^2)\times_{G_n^2}\widetilde{G}_n^2$.
\item
The mapping $\kappa\colon\widetilde{P}^2\to P^2$ is well-defined bundle morphism as $G_n^2$-bundles.
If $(U,\varphi)$ is a chart about $\pi(u(o))$ and if we represent $u$ as $(h^i,h^i{}_j,h^i{}_{jk})$, then $\kappa(u)$ is represented as $\left(h^i,h^i{}_j,\frac{h^i{}_{jk}+h^i{}_{kj}}2\right)$.
\item
When regarded as a mapping from $\widetilde{G}_n^2$ to $G_n^2$, $\kappa$ is a homomorphism.
\end{enumerate}
\end{theorem}

\begin{remark}
\label{rem1.14}
If we replace $\GL_n(\mathbb{R})$ with a proper Lie subgroup, then the above averaging procedure does not work in general.
For example, if $G=\SL_n(\mathbb{R})$, then we have $h^i{}_{ik}=0$ but mostly we do \textit{not} have $h^i{}_{ji}=0$.
Even if the latter condition is satisfied, it is still not obvious that such a formal frame can be realized by a volume preserving local diffeomorphism.
On the other hand, we can replace $\mathbb{R}$ by $\mathbb{C}$, moreover, we can work in the holomorphic category so long as we stay in formal frames of finite order.
See also Remarks~\ref{rem4.9} and~\ref{rem5.9}.
\end{remark}

Let $u\in\widetilde{P}^2$ and we represent $u=j^1_o(F)$.
Then $F$ induces an isomorphism from $T_{\id}P^1(\mathbb{R}^n)$ to $T_u P^1$, where $\id$ denotes the identity map.
Indeed, let $X\in T_{\id}P^1(\mathbb{R})$ and represent $X$ by a family of local diffeomorphisms $g_t\colon(\mathbb{R}^n,o)\to\mathbb{R}^n$ such that $g_0=\id$.
We can represent $X$ as $(X^i,X^i{}_j)$ by considering the natural coordinates.
Then, we have $\left.\pdif{g_t}{t}(o)\right|_{t=0}=X^i$ and that $\left.\pdif{Dg_t}{t}(o)\right|_{t=0}=X^i{}_j$.
We have
\stepcounter{theorem}
\begin{align*}
\tag{\thetheorem-1}
\label{eqF10-1}
&\hphantom{{}={}}%
\left.\pdif{}{t}\right|_{t=0}f\circ g_t(o)=Df(o)\left.\pdif{g_t}{t}\right|_{t=0}(o)=Df(o)^i{}_lX^l=F(o)^i{}_\alpha X^\alpha,\\*
\tag{\thetheorem-2}
\label{eqF10-2}
&\hphantom{{}={}}%
\left.\pdif{}{t}\right|_{t=0}(F\circ g_t)(o)Dg_t(o)\\*
&=DF(o)^i{}_{\alpha\beta}Dg_0(o)^\alpha{}_j\pdif{g_t}{t}(o)^\beta+F(o)^i{}_\alpha\left.\pdif{Dg_t}{t}(o)^\alpha{}_j\right|_{t=0}\\*
&=DF(o)^i{}_{j\alpha}X^\alpha+F(o)^i{}_\alpha X^\alpha{}_j.
\end{align*}

\begin{definition}
\label{defF11}
We represent the isomorphism from $T_{\id}P^1(\mathbb{R}^n)$ to $T_{\pi^1(u)}P^1$ obtained as above again by $u$ by abuse of notations.
\end{definition}

Similarly, we have an adjoint action of $\widetilde{G}_n^2$ on $T_{\id}P^1(\mathbb{R}^n)$.
Let $a\in\widetilde{G}_n^2$ and $X\in T_{\id}P^1(\mathbb{R}^n)$.
We represent $a$ by $F$ and $X$ by $g_t$, respectively.
Then, $\Ad_{a^{-1}}X$ is by definition the vector represented by $(f^{-1}\circ g_t\circ f,(F^{-1}\circ g_t\circ f)(Dg_t\circ f)F)$.
Concretely, we have
\begin{align*}
&\hphantom{{}={}}
\left.\pdif{}{t}\right|_{t=0}(f^{-1}\circ g_t\circ f)(o)\\*
&=(Df(o)^{-1})^i{}_lX^l,\\*
&\hphantom{{}={}}
\left.\pdif{}{t}\right|_{t=0}(F^{-1}\circ g_t\circ f)(o)(Dg_t\circ f)(o)F(o)\\*
&=-(F(o)^{-1})^i{}_\alpha DF(o)^\alpha{}_{\beta\gamma}(F(o)^{-1})^\beta{}_\delta X^\gamma F(o)^\delta{}_j+F^{-1}(o)^i{}_\alpha X^\alpha{}_\beta F(o)^\beta{}_j\\*
&=-(F(o)^{-1})^i{}_\alpha DF(o)^\alpha{}_{j\beta}X^\beta+F^{-1}(o)^i{}_\alpha X^\alpha{}_\beta F(o)^\beta{}_j.
\end{align*}

\section{The canonical form on the bundle of formal $2$-frames}
We are now in position to introduce the canonical form on $\widetilde{P}^2$ and its fundamental properties.

\begin{definition}
\label{defG1}
Let $u\in\widetilde{P}^2$ and $X\in T_u\widetilde{P}^2$.
We set $\theta(X)=u^{-1}(\pi^2_1{}_*X)\in T_{\id}P^1(\mathbb{R}^n)$ and call $\theta$ the \textit{canonical form}.
\end{definition}

\begin{definition}
Let $G$ be a Lie group.
Let $P$ be a principal $G$-bundle and $u\in P$.
We consider $u$ as a mapping from $G$ to $P$.
If $X\in\mathfrak{g}$, then we set $X^*_u=u_*X$.
We call $X^*$ the fundamental vector field associated with $X$.
\end{definition}

The following theorems directly follow from definitions.

\begin{theorem}
\label{thmG2}
We have the following.
\begin{enumerate}
\item
If $a\in\widetilde{G}_n^2$, then $R_a^*\theta=\Ad_{a^{-1}}\theta$, where $R_a$ and $\Ad_a$ denote the right action and the adjoint action of\/ $a$, respectively.
\item
If $X\in\widetilde{\mathfrak{g}}_n^2$, then $\theta^2(X^*)=\pi^2_1{}_*X$.
\end{enumerate}
\end{theorem}

\begin{theorem}
\label{thmG3}
Let $M,N$ be manifolds and $f\colon M\to N$ be a local diffeomorphism.
\begin{enumerate}
\item
We have $f^*\widetilde{P}^2(N)=\widetilde{P}^2(M)$.
\item
If $\theta_M,\theta_N$ denote the canonical forms on $M,N$, then we have $f^*\theta_N=\theta_M$.
\end{enumerate}
\end{theorem}

Note that $\theta$ is naturally represented as $(\theta^i,\theta^i{}_j)$.
We set $\theta^0=(\theta^i)$ and $\theta^1=(\theta^i{}_j)$.

\begin{definition}
We call $\theta^i$ the \textit{canonical form} of order $i+1$.
We set
\begin{align*}
\Theta&=d\theta^0+\theta^1\wedge\theta^0,\\*
\Omega&=d\theta^1+\theta^1\wedge\theta^1,
\end{align*}
and call them the \textit{torsion} and the \textit{curvature} of $\theta$.
\end{definition}

The canonical form is locally represented as follows by~\eqref{eqF10-2}.

\begin{lemma}
\label{lem2.6}
Let $(U,\varphi)$ be a chart and $u=(u^i,u^i{}_j,u^i{}_{jk})$ the associated natural coordinates.
If we set $(v^i{}_j)=(u^i{}_j)^{-1}$, then we have
\begin{align*}
\theta^0{}_u(X)&=v^i{}_\alpha du^\alpha,\\*
\theta^1{}_u(X)&=v^i{}_\alpha du^\alpha{}_j-v^i{}_\alpha u^\alpha{}_{j\beta}v^\beta{}_\gamma du^\gamma.
\end{align*}
\end{lemma}
The proof is straightforward and omitted.

\begin{theorem}
\label{thm2.6}
We have $P^2=\{u\in\widetilde{P}^2\mid \Theta_u=0\}$.
\end{theorem}
\begin{proof}
First note that we have $dv^i{}_j=-v^i{}_\alpha du^\alpha{}_\beta v^\beta{}_j$.
Hence we have
\begin{align*}
\Theta^i&=dv^i{}_\alpha\wedge du^\alpha+v^i{}_\alpha du^\alpha{}_\beta\wedge v^\beta{}_\gamma du^\gamma-v^i{}_\alpha u^\alpha{}_{\gamma\beta}v^\beta{}_\delta du^\delta\wedge v^\gamma{}_\epsilon du^\epsilon\\*
&=-v^i{}_\alpha du^\alpha{}_\beta v^\beta{}_\gamma\wedge du^\gamma+v^i{}_\alpha du^\alpha{}_\beta\wedge v^\beta{}_\gamma du^\gamma+v^i{}_\alpha u^\alpha{}_{\beta_1\beta_2}v^{\beta_1}{}_{\gamma_1}du^{\gamma_1}\wedge v^{\beta_2}{}_{\gamma_2}du^{\gamma_2}\\*
&=v^i{}_\alpha u^\alpha{}_{\beta_1\beta_2}v^{\beta_1}{}_{\gamma_1}du^{\gamma_1}\wedge v^{\beta_2}{}_{\gamma_2}du^{\gamma_2}.
\end{align*}
Therefore, $\Theta_u=0$ if and only if $u^\alpha{}_{\beta_1\beta_2}=u^\alpha{}_{\beta_2\beta_1}$, that is, $u\in P^2$.
\end{proof}

\begin{remark}
We have a kind of split exact sequence
\[
\xymatrix@1@C=30pt@M=6pt{
0 \ar[r] & P^2 \ar@<0.5ex>[r]^{\epsilon} & \widetilde{P}^2 \ar@<0,5ex>[l]^{\kappa} \ar[r]^-{\Theta} & M\times\mathbb{R}^n \ar[r] & 0.
}
\]
A generalization of Theorem~\ref{thm2.6} is given as Theorem~\ref{thm5.7}.
\end{remark}

The curvature $\Omega$ of $\theta$ is calculated as follows.

\begin{lemma}
We locally have
\begin{align*}
\Omega^i{}_j&=-v^i{}_\alpha du^\alpha{}_{j\beta}\wedge v^\beta{}_\gamma du^\gamma\\*
&\hphantom{{}={}}%
+v^i{}_\alpha u^\alpha{}_{j\beta}v^\beta{}_\gamma du^\gamma{}_\delta\wedge v^\delta{}_\epsilon du^\epsilon%
+v^i{}_\alpha u^\alpha{}_{\beta_1\beta_2}v^{\beta_1}{}_{\gamma_1}du^{\gamma_1}{}_j\wedge v^{\beta_2}{}_{\gamma_2}du^{\gamma_2}\\*
&\hphantom{{}={}}%
-v^i{}_\alpha u^\alpha{}_{\beta_1\beta_2}v^{\beta_1}{}_{\gamma_1}u^{\gamma_1}{}_{j\delta}v^\delta{}_\epsilon du^\epsilon\wedge v^{\beta_2}{}_{\gamma_2}du^{\gamma_2}.
\end{align*}
\end{lemma}
\begin{proof}
We have
\begin{align*}
\Omega^i{}_j&=dv^i{}_\alpha\wedge du^\alpha{}_j\\*
&\hphantom{{}={}}%
-dv^i{}_\alpha u^\alpha{}_{j\beta}v^\beta{}_\gamma\wedge du^\gamma-v^i{}_\alpha du^\alpha{}_{j\beta}v^\beta{}_\gamma\wedge du^\gamma-v^i{}_\alpha u^\alpha{}_{j\beta}dv^\beta{}_\gamma\wedge du^\gamma\\*
&\hphantom{{}={}}%
+v^i{}_\alpha du^\alpha{}_\beta\wedge v^\beta{}_\gamma du^\gamma{}_j\\*
&\hphantom{{}={}}%
-v^i{}_\alpha du^\alpha{}_\beta\wedge v^\beta{}_\gamma u^\gamma{}_{j\delta}v^\delta{}_\epsilon du^\epsilon-v^i{}_\alpha u^\alpha{}_{\beta_2\beta_1}v^{\beta_1}{}_{\gamma_1}du^{\gamma_1}\wedge v^{\beta_2}{}_{\gamma_2}du^{\gamma_2}{}_j\\*
&\hphantom{{}={}}%
+v^i{}_\alpha u^\alpha{}_{\beta_2\beta_1}v^{\beta_1}{}_{\gamma_1}du^{\gamma_1}\wedge v^{\beta_2}{}_{\gamma_2}u^{\gamma_2}{}_{j\delta_2}v^{\delta_2}{}_{\epsilon_2}du^{\epsilon_2}.
\end{align*}
The first and the fifth, the second and the sixth terms cancel each other.
By rearranging the indices, we obtain the result.
\end{proof}

\begin{remark}
Both the torsion $\Theta$ and the curvature $\Omega$ involve $\theta^0$ rather than $du^i$.
We do not have characterizations of the curvature, however, it is related with the torsion of order $2$.
See Example~\ref{ex5.3}.
On the other hand, torsions have a clear meaning.
See Theorem~\ref{thm5.7}.
\end{remark}

\begin{remark}
We will introduce Lie groups $\widetilde{G}_n^r$ in Section~\ref{sec_Higheroder}, and there will be natural projections from $\widetilde{G}_n^r$ to $\widetilde{G}_n^{r-1}$.
On the other hand, we have an inclusion $\GL_n(\mathbb{R})$ into $\widetilde{G}_n^2$ defined by $a=(a^i{}_j)\mapsto(a^i{}_j,0)$.
If $u=(u^i,u^i{}_j,u^i{}_{jk})$ are the natural coordinates associated with a chart and if $a\in\GL_n(\mathbb{R})$, then we have $u.a=(u^i,u^i{}_la^l{}_j,u^i{}_{lm}a^l{}_ja^m{}_k)$ by~\eqref{eqF4}.
\end{remark}

We recall the notion of connections.

\begin{definition}
A $\gl_n(\mathbb{R})$-valued $1$-form $\theta$ on $P^1$ is said to be a \textit{connection} if we have the following:
\begin{enumerate}
\item
If $a\in\GL_n(\mathbb{R})$, then $R_a^*\theta=\Ad_{a^{-1}}\omega$.
\item
If $X\in\gl_n(\mathbb{R})$ and if $X^*$ denotes the fundamental vector field associated with $X$, then $\theta(X^*)=X$.
\end{enumerate}
\end{definition}

As in the classical cases, we have the following

\begin{theorem}[cf. Proposition~7.1 (p.~147) of~\cite{K}, Theorems~2 and~3 of~\cite{Garcia}]
\label{thmG11}
\ \linebreak
There is a one to one correspondence between the following objects\/\textup{:}
\begin{enumerate}
\item
Connections on $P^1$.
\item
Sections of $\pi^2_1\colon\widetilde{P}^2\to P^1$ which are equivariant under the $\GL_n(\mathbb{R})$-actions.
\item
Sections of $\widetilde{P}^2/\GL_n(\mathbb{R})\to M$.
\end{enumerate}
\end{theorem}

Before proving Theorem~\ref{thmG11}, we show the following

\begin{lemma}
\label{lemG12}
Sections of $\widetilde{P}^2/\GL_n(\mathbb{R})\to M$ is in one to one correspondence between $\GL_n(\mathbb{R})$-equivariant mappings from $\widetilde{P}^2$ to $\widetilde{G}_n^2/\GL_n(\mathbb{R})$, where $\widetilde{G}_n^2$ acts on $\widetilde{G}_n^2/\GL_n(\mathbb{R})$ on the right by $[g].a=[a^{-1}g]$.
\end{lemma}
\begin{proof}
We repeat a proof in Husem\"oller~\cite{Husemoller} for convenience.
First note that $\widetilde{P}^2/\GL_n(\mathbb{R})=\widetilde{P}^2\times_{\widetilde{G}_n^2}\widetilde{G}_n^2/\GL_n(\mathbb{R})$.
We represent elements of $\widetilde{P}^2\times_{\widetilde{G}_n^2}\widetilde{G}_n^2/\GL_n(\mathbb{R})$ as $[u,\alpha]$.
Let $\sigma\colon M\to\widetilde{P}^2/\GL_n(\mathbb{R})$ be a section.
If $u\in\widetilde{P}^2$, there uniquely exists an element, say $\alpha(u)\in\widetilde{G}_n^2/\GL_n(\mathbb{R})$, such that $\sigma(\pi^2(u))=[u,\alpha(u)]$.
If $a\in\widetilde{G}_n^2$, then we have $[u,\alpha(u)]=[u.a,\alpha(u).a]$.
On the other hand, we have $\sigma(\pi^2(u))=\sigma(\pi^2(u.a))=[u.a,\alpha(u.a)]$ so that $\alpha(u.a)=\alpha(u).a$.
Suppose conversely that $\alpha$ is given.
If $p\in M$, then we choose $u\in(\pi^2)^{-1}(p)$ and set $\sigma(p)=[u,\alpha(u)]\in\widetilde{P}^2/\GL_n(\mathbb{R})$.
If $a\in\widetilde{G}_n^2$, then we have $[u.a,\alpha(u.a)]=[u.a,\alpha(u).a]=[u,\alpha(u)]$ so that $\sigma$ is a well-defined section.
It is easy to see that this correspondence is one~to~one.
\end{proof}

\begin{proof}[Proof of Theorem~\ref{thmG11}]
The proof is almost identical to that of Proposition~7.1 of \cite{K}.
First let $\omega$ be a connection.
Let $(U,\varphi)$ be a chart and consider the associated natural coordinates $(u^i,u^i{}_j,u^i{}_{jk})$.
Then, $(u^i,u^i{}_j)$ are local coordinates for $P^1$.
Let $s$ be the local trivialization of $P^1$ which corresponds to these coordinates and set $\mu=s^*\omega$.
We represent $\mu=\pdif{}{x^i}\Gamma^i{}_{jk}dx^k$ using the Christoffel symbols.
Note that our symbol differs from the usual one, that is, the order of lower indices are reversed.
We set $\sigma(u^i,u^i{}_j)=(u^i,u^i{}_j,-\Gamma^i{}_{\alpha\beta}u^\alpha{}_ju^\beta{}_k)$.
If we set $\psi=(D\phi)^{-1}$, then we have
\stepcounter{theorem}
\[
\widehat{\Gamma}^i{}_{jk}=-H\phi^i{}_{\alpha\beta}\psi^\alpha{}_j\psi^\beta{}_k+D\phi^i{}_\alpha\Gamma^\alpha{}_{\beta\gamma}\psi^\beta{}_j\psi^\gamma{}_k.
\tag{\thetheorem}
\label{eq2.14}
\]
Hence we have
\begin{align*}
\widehat{\Gamma}^i{}_{\alpha\beta}\widehat{u}^\alpha{}_j\widehat{u}^\beta{}_k%
&=-H\phi^i{}_{\alpha\beta}\psi^\alpha{}_\gamma\widehat{u}^\gamma{}_j\psi^\beta{}_\delta\widehat{u}^\delta{}_k+D\phi^i{}_\alpha\Gamma^\alpha{}_{\beta\gamma}\psi^\beta{}_\delta\widehat{u}^\delta{}_j\psi^\gamma{}_\epsilon\widehat{u}^\epsilon{}_k\\*
&=-H\phi^i{}_{\alpha\beta}u^\alpha{}_ju^\beta{}_k+D\phi^i{}_\alpha\Gamma^\alpha{}_{\beta\gamma}u^\beta{}_ju^\gamma{}_k.
\end{align*}
It follows that
\[
(\widehat{u}^i{}_j,-\widehat{\Gamma}^i{}_{\alpha\beta}\widehat{u}^\alpha{}_j\widehat{u}^\beta{}_k)=(D\phi^i{}_j,H\phi^i{}_{jk})(u^i{}_j,-\Gamma^i{}_{\alpha\beta}u^\alpha{}_ju^\beta{}_k).
\]
Therefore, locally defined $\sigma$ gives rise to a well-defined section of $\pi^2_1$ (cf.~\eqref{eqF8}).
If $a\in\GL_n(\mathbb{R})$, then we have
\begin{align*}
\sigma((u^i,u^i{}_j).a)&=\sigma(u^i,u^i{}_\alpha a^\alpha{}_j)\\*
&=(u^i,u^i{}_\alpha a^\alpha{}_j,-\Gamma^i{}_{\alpha\beta}u^\alpha{}_\gamma u^\beta{}_\delta a^\gamma{}_ja^\delta{}_k)\\*
&=\sigma(u^i,u^i{}_j).a
\end{align*}
so that the section is $\GL_n(\mathbb{R})$-equivariant.
Conversely, if $s$ is a $\GL_n(\mathbb{R})$-equivariant section of $\pi^2_1$, then $s^*\theta^2$ is a connection.

Next, let $\sigma\colon M\to\widetilde{P}^2/\GL_n(\mathbb{R})$ be a section.
By Lemma~\ref{lemG12}, $\sigma$ corresponds to a $\widetilde{G}_n^2$-equivariant map, say $\alpha$, from $\widetilde{P}^2$ to $\widetilde{G}_n^2/\GL_n(\mathbb{R})$.
We set $P=\alpha^{-1}([e])$, where $e\in\widetilde{G}_n^2$ denotes the unit.
Then $P$ is naturally a principal $\GL_n(\mathbb{R})$-bundle and the inclusion gives the desired section.
As $\widetilde{G}_n^2/\GL_n(\mathbb{R})\cong\mathbb{R}^{n^3}$ is contractible, thus obtained $P$ is isomorphic to each other.
Since there is a connection on $P^1$, we have a section from $P^1\to\widetilde{P}^2$ so that $P$ is isomorphic to $P^1$.
Conversely, let $\sigma\colon P^1\to\widetilde{P}^2$ be a $\GL_n(\mathbb{R})$-equivariant section.
By taking the quotients by $\GL_n(\mathbb{R})$-actions, we obtain a section $M\to\widetilde{P}^2/\GL_n(\mathbb{R})$.
We omit to show that these correspondences are one~to~one.
\end{proof}

\begin{remark}
Connections on $P^1$ correspond to affine connections on $TM$, and $\GL_n(\mathbb{R})$-equivariant sections of $\pi^2_1$ correspond to reductions of $\widetilde{P}^2$ to $\GL_n(\mathbb{R})$.
\end{remark}

\begin{remark}
The section associated with a connection corresponds to the geodesic equation although the lower indices of the Christoffel symbols need not commute.
\end{remark}

\begin{remark}
If $\nabla$ is a linear connection, then we can find a torsion free connection by modifying $\nabla$~\cite{KN}*{Proposition 7.9 (Chapter~3)}.
We can interpret this procedure as considering $\kappa_*\nabla$ instead of $\nabla$, where $\kappa_*\colon\widetilde{\mathfrak{g}}_n^2\to\mathfrak{g}_n^2$.
\end{remark}

\section{Comparison with Garc\'\i a's construction}
Let $G$ be a Lie group with Lie algebra $\mathfrak{g}$.
Let $M$ be a manifold and $\pi\colon P\to M$ a principal $G$-bundle over $M$.
Gac\'\i a gives in~\cite{Garcia} canonical forms on $1$-jet bundles associated with principal $G$-bundles.

\begin{definition}
Let $\mathcal{J}(P)$ the bundle of $1$-jets of germs of sections from $M$ to $P$.
If $s_p\in\mathcal{J}(P)$ is a jet at $p\in M$, then we set $\pi^1(s_p)=s_p(p)$ and $\overline{\pi}^1(s_p)=p$.
The bundle $\mathcal{J}(P)$ over $P$ is called the \textit{$\mathit{1}$-jet bundle} of $P$.
\end{definition}

Let $(U,\varphi)$ be a chart on $M$.
We assume that $P$ is trivial on $U$ and let $\psi\colon\pi^{-1}(U)\to U\times G$ be a trivialization.
Let $s$ be a section of $\pi\colon P\to M$ on $U\subset M$.
We represent $s$ by $\psi\circ s\circ\varphi^{-1}$ and $j^1_p(s)$ by $j^1_x(\psi\circ s\circ\varphi^{-1})$, where $x=\varphi(p)\in\varphi(U)$.
More concretely, let $\psi\circ s\circ\varphi^{-1}=(\id,h)$, where $h$ is a $G$-valued function on $\varphi(U)$.
Then, $j^1_p(s)$ is represented by $(x,h(x),Dh(x))$, which are the coordinates used by Garc\'\i a.

The following is obvious.
\begin{lemma}
The bundle $\mathcal{J}(P)$ admits a natural $G$-action on the right.
\end{lemma}

Let $u\in P$, $Y\in T_uP$ and suppose that $\pi_*Y=0$.
Then, there uniquely exists an element $Z\in\mathfrak{g}$ such that $Z^*_u=Y$, where $Z^*$ denotes the fundamental vector field associated with $Z$.
We represent $Z$ by $u^{-1}Y$.

\begin{definition}[Canonical form]
Let $p\in M$, $j^1_p(s)\in\mathcal{J}(P)$ and $X\in T_{j^1_p(s)}\mathcal{J}(P)$.
We set
\begin{align*}
d^\nu{}_{j^1_p(s)}X&=\pi^1{}_*X-s_*\overline{\pi}^1_*X,\\*
\theta_{j^1_p(s)}(X)&=s(p)^{-1}(d^\nu{}_{j^1_p(s)}X).
\end{align*}
We call $\theta$ the \textit{canonical form} on $\mathcal{J}(P)$.
\end{definition}

\begin{proposition}[cf. Theorem~\ref{thmG2}]
We have the following\textup{:}
\begin{enumerate}[\textup{\theenumi)}]
\item
If $g\in G$, then $R_g^*\theta=\Ad_{g^{-1}}\theta$.
\item
If $X\in\mathfrak{g}$ and if $X^*$ denotes the fundamental vector field associated with $X$, then $\theta(X^*)=\pi(X)$.
\end{enumerate}
\end{proposition}

In what follows, we assume that $\dim M=n$ and that $P$ is the frame bundle of $M$.
We have $G=\GL_n(\mathbb{R})$.
It might happen that $P$ admits a reduction, that is, $G$ might be a proper Lie subgroup of $G$.
Some of arguments work in such cases.

Let $j^2_o(f)\in P^2$ and $p=f(o)$.
It is easy to see that $j^1_p(Df\circ f^{-1})$ is well-defined.
In general, if $j^1_o(F)\in\widetilde{P}^2$, then $j^1_p(F\circ f^{-1})$ is also well-defined, where $p=f(o)$.
We set $\Phi(j^1_o(F))=j^1_p(F\circ f^{-1})$.
Conversely, let $p\in M$ and $s_p\in\mathcal{J}(P)$.
We regard $s(p)$ as a linear map from $T_o\mathbb{R}^n$ to $T_pM$ and choose a local diffeomorphism $f\colon(\mathbb{R}^n,o)\to(M,p)$ such that $Df(o)=s(p)$.
We consider the pair $(f,s\circ f)$ as a local isomorphism from $T(\mathbb{R}^n,o)$ to $T(M,p)$ and set $\Psi(s_p)=j^1_o(f,s\circ f)$.

\begin{lemma}
\label{lem3.5}
The mapping $\Psi$ is well-defined and we have $\Psi=\Phi^{-1}$.
\end{lemma}
\begin{proof}
First, $j^1_o(f,s\circ f)\in\widetilde{P}^2$ because $s\circ f(o)=s(p)=Df(o)$.
We have $j^1_o(f,s\circ f)=(f(o),s\circ f(o),Ds(f(o))Df(o))=(p,s(p),Ds(p))$ so that $\Psi$ is well-defined.
Finally, we have $\Psi=\Phi^{-1}$ by the definitions.
\end{proof}

If $(U,\varphi)$ is a chart about $p$, then mappings $\Phi$ and $\Psi$ are represented as follows.
First, if $j^1_o(F)\in\widetilde{P}^2$, then we set $F_\varphi=D\varphi\circ F$ and $f_\varphi=\varphi\circ f$.
If we set $q=\varphi(p)$, then $j^1_o(F)$ is represented by $(q,F_\varphi(o),DF_\varphi(o))$.
On the other hand, $j^1_p(F\circ f^{-1})$ is represented~by
\begin{align*}
j^1_q(F_\varphi\circ f_\varphi{}^{-1})&=(q,F_\varphi(o),DF_\varphi(o)(Df_\varphi(o))^{-1})\\*
&=(q,F_\varphi(o),DF_\varphi(o)^i{}_{j\alpha}((F_\varphi(o))^{-1})^\alpha{}_k)
\end{align*}
in the Garc\'\i a coordinates.
Therefore, if we make use of the natural coordinates on $\widetilde{P}^2$ and the Garc\'\i a coordinates on $\mathcal{J}(P)$, then we have
\begin{align*}
\Phi(u^i,u^i{}_j,u^i{}_{jk})&=(u^i,u^i{}_j,u^i{}_{jl}v^l{}_k),\\*
\Psi(x^i,y^i{}_j,z^i{}_{jk})&=(x^i,y^i{}_j,z^i{}_{jl}y^l{}_k),
\end{align*}
where $(v^i{}_j)=(u^i{}_j)^{-1}$.
On the other hand, $\widetilde{G}_n^2$ acts on $\mathcal{J}(P)$ on the right by Lemma~\ref{lem3.5}.
Indeed, if $s_p\in\mathcal{J}(P)$ and if $a\in\widetilde{G}_n^2$, then we can set $(s_p).a=\Phi(\Psi(s_p).a)$.
In the Garc\'\i a coordinates, we have
\[
(x^i,y^i{}_j,z^i{}_{jk})(a^i{}_j,a^i{}_{jk})=
(x^i,y^i{}_\alpha a^\alpha{}_j,z^i{}_{\alpha k}a^\alpha{}_j+y^i{}_\alpha a^\alpha{}_{j\beta}b^\beta{}_\gamma w^\gamma{}_k)
\]
where $(w^i{}_j)=(y^i{}_j)^{-1}$ and $(b^i{}_j)=(a^i{}_j)^{-1}$.
If $a\in\GL_n(\mathbb{R})$, namely, if $a^i{}_{jk}=0$, then we have
\[
(x^i,y^i{}_j,z^i{}_{jk})(a^i{}_j,a^i{}_{jk})=(x^i,y^i{}_\alpha a^\alpha{}_j,z^i{}_{\alpha k}a^\alpha{}_j).
\]
Hence the $\widetilde{G}_n^2$-action on $\mathcal{J}(P)$ is an extension of the $\GL_n(\mathbb{R})$-action.

The following is known.
We modify notations fitting to ours.

\begin{lemma}[\cite{Garcia}*{Section~3 and Lemma~1}]
Let $\theta'$ be the canonical form on $\mathcal{J}(P)$.
We have
\[
\theta'=w^i{}_\alpha(dy^\alpha{}_j-y^\alpha{}_{j\beta}dx^\beta).
\]
\end{lemma}

Therefore, we have the following.

\begin{theorem}
\begin{enumerate}
\item
The mapping $\Phi$ is an isomorphism of $\GL_n(\mathbb{R})$-bundles.
\item
If we define a $\widetilde{G}_n^2$-action on $\mathcal{J}(P)$ as above, then $\mathcal{J}(P)$ is a principal $\widetilde{G}_n^2$-bundle which is compatible the original $\GL_n(\mathbb{R})$-action on~$\mathcal{J}(P)$.
Moreover, $\Phi$ is an isomorphism of $\widetilde{G}_n^2$-bundles under these actions.
\item
If $\theta^1$ and $\theta'$ denote the canonical form of order~$2$ on $\widetilde{P}^2$ and the canonical form on $\mathcal{J}(P)$, respectively, then we have $\Phi^*\theta'=\theta^1$.
\end{enumerate}
\end{theorem}

\section{Bundles of formal frames of higher order}
\label{sec_Higheroder}
We can consider analogues of $P^r$, $r\geq3$.
We set $\widetilde{P}^0=M$, $\widetilde{P}^1=P^1$ and let $\pi^1_0$ denote the projection from $P^1$ to $M$.
Suppose that groups $\widetilde{G}^k_n$, $\widetilde{G}^k_n$-bundles $\widetilde{P}^k$ such that $\widetilde{G}^k_n=\widetilde{P}^k(\mathbb{R}^n,o)$, and projections $\pi^k_{k-1}\colon\widetilde{P}^k\to\widetilde{P}^{k-1}$ equivariant under the $\widetilde{G}^k_n$ and $\widetilde{G}^{k-1}_n$ actions are defined up to $k=r-1$.
This holds true for $r=2$.
Let $F=F^{r-1}\colon\widetilde{P}^{r-1}(\mathbb{R}^n,o)\to\widetilde{P}^{r-1}(M)$ be a locally defined isomorphism of $\widetilde{G}^{r-1}$-bundles and $F^k\colon\widetilde{P}^k(\mathbb{R}^n,o)\to\widetilde{P}^k(M)$, where $0\leq k\leq r-2$, be the underlying isomorphism in the sense that $F^0,F^1,\ldots,F^{r-1}$ are bundle isomorphisms and that
\[
\begin{CD}
\widetilde{P}^k(\mathbb{R}^n,o) @>{F^k}>> \widetilde{P}^k(M)\\
@V{\pi^k_l}VV @VV{\pi^k_l}V\\
\widetilde{P}^l(\mathbb{R}^n,o) @>>{F^l}> \widetilde{P}^l(M)\rlap{,}
\end{CD}
\]
where $\pi^k_l=\pi^{l+1}_l\circ\cdots\circ\pi^k_{k-1}$, is commutative for $1\leq l<k\leq r-1$.

\begin{definition}[cf. Definition~\ref{def1.6}]
We set
\begin{align*}
\widetilde{P}^r(M)&=\left\{j^1_o(F)\;\middle|\;\parbox[c]{160pt}{$F\colon\widetilde{P}^{r-1}(\mathbb{R}^n,o)\to\widetilde{P}^{r-1}(M)$,\\ $F^k(o)=j^1_o(F^{k-1})$ for $1\leq k\leq r$}\right\},\\*
\widetilde{G}_n^r&=\widetilde{P}^r(\mathbb{R}^n,o).
\end{align*}
We call $\widetilde{P}^r(M)$ as the \textit{bundle of formal frames} of order $r$, and elements of $\widetilde{P}^r(M)$ \textit{formal frames} of order $r$.
If $F\in\widetilde{P}^r(M)$, then we set $\pi^r_{r-1}(j^1_o(F))=F(o)$.
\end{definition}
We call formal frames of order greater than two as \textit{formal frames of higher order}.

We have the following.
The proof is easy and omitted.
\begin{lemma}
\begin{enumerate}
\item
Thus defined $\widetilde{G}_n^r$ is a Lie group of dimension $\sum_{l=1}^rn^l$.
\item
If we set $\pi^r=\pi^{r-1}\circ\pi^r_{r-1}$, then $\pi^r\colon\widetilde{P}^r(M)\to M$ is a principal $\widetilde{G}_n^r$-bundle.
\end{enumerate}
\end{lemma}

\begin{definition}
If $u=j^r_o(f)\in P^r(M)$, then we set $\epsilon(u)=j^1_o(D^{r-1}(f))$.
\end{definition}

Note that $\epsilon$ gives rise to an embedding of $G_n^r$ into $\widetilde{G}_n^r$.

\begin{definition}[cf. Definition~\ref{defG1}]
\label{def4.4}
Let $u\in\widetilde{P}^r(M)$ and $X\in T_u\widetilde{P}^r(M)$.
We set $\theta(X)=u^{-1}(\pi^r_{r-1*}X)\in T_{\id}\widetilde{P}^{r-1}(\mathbb{R}^n)$, where $u\colon T_{\id}\widetilde{P}^{r-1}(\mathbb{R})\to T_{\pi^r_{r-1}(u)}\widetilde{P}^{r-1}(M)$ is the isomorphism induced from the right action (cf. Definition~\ref{defF11}).
We call $\theta$ the \textit{canonical form}.
The canonical form is naturally represented as $(\theta^i,\theta^i{}_{j_1},\ldots,\theta^i{}_{j_1,\ldots,j_{r-1}})$.
We refer to $(\theta^i{}_{j_1,\ldots,j_{r-1}})$ as the canonical form of order $r$.
\end{definition}

The adjoint action of $\widetilde{G}_n^r$ on $T_{\id}\widetilde{P}^{r-1}(\mathbb{R}^n)$ is defined in a similar way as in the case of $r=2$.

\begin{theorem}[cf. Theorem~\ref{thmG2}]
We have the following.
\begin{enumerate}
\item
If $a\in\widetilde{G}_n^r$, then $R_a^*\theta=\Ad_{a^{-1}}\theta$.
\item
If $X\in\widetilde{\mathfrak{g}}_n^r$, then $\theta^r(X^*)=\pi^r_{r-1}{}_*X$, where $X^*$ denotes the fundamental vector field associated with $X$.
\end{enumerate}
\end{theorem}

The product in $\widetilde{G}_n^r$ remains similar to that of the group of $r$-frames $G_n^r$.
For example, if $r=3$, the product is given as follows.
Let $a,b\in\widetilde{G}_n^3$ and $a=j^1_o(F)$, $b=j^1_o(G)$.
Then, $ab=j^1_o(F\circ G)$.
If we represent $a$ as $a=(a^i{}_j,a^i{}_{jk},a^i{}_{jkl})$ and so on, we have
\begin{align*}
(ab)^i{}_j&=a^i{}_\alpha b^\alpha{}_j,\\*
(ab)^i{}_{jk}&=a^i{}_{\alpha\beta}b^\alpha{}_jb^\beta{}_k+a^i{}_\alpha b^\alpha{}_{jk},\\*
(ab)^i{}_{jkl}&=a^i{}_{\alpha\beta\gamma}b^{\alpha}{}_jb^{\beta}{}_kb^{\gamma}{}_l+a^i{}_{\alpha\beta}b^{\alpha}{}_{jl}b^{\beta}{}_k+a^i{}_{\alpha\beta}b^\alpha{}_jb^\beta{}_{kl}+a^i{}_{\alpha\beta}b^\alpha{}_{jk}b^\beta{}_l+a^i{}_{\alpha}b^{\alpha}{}_{jkl}.
\end{align*}
These formulae are obtained as follows.
Let $a,b\in\widetilde{G}_n^3$ and $a=j^1_o(F), b=j^1_o(G)$.
First, we represent $F(x)=(f^i(x),f^i{}_j(x),f^i{}_{jk}(x),f^i{}_{jkl}(x))$.
We have $F^0(x)=(f^i(x))$, $F^1(x)=(f^i(x),f^i{}_j(x))$, $F^2(x)=(f^i(x),f^i{}_j(x),f^i{}_{jk}(x))$ and $F^3(x)=F(x)$.
By the conditions required to $F$, we have
\begin{align*}
a^i{}_j&=\pdif{f^i}{x^j}(o)=f^i{}_j(o),\\*
a^i{}_{jk}&=\pdif{f^i{}_j}{x^k}(o)=f^i{}_{jk}(o),\\*
a^i{}_{jkl}&=\pdif{f^i{}_{jk}}{x^l}(o).
\end{align*}
We have $(F\circ G)^0(x)=f^i(g^m(x))$ so that $(ab)^i{}_j=a^i{}_\alpha b^\alpha{}_j$ holds in $\widetilde{G}_n^1$.
Hence we may assume that $(F\circ G)^1(x)=(f^i(g^m(x)),f^i{}_\alpha(g^m(x))g^\alpha{}_j(x))$ because we are concerned with jets at $o$.
It follows that
\[
(ab)^i{}_{jk}=(a^i{}_{\alpha\beta}b^\alpha{}_jb^\beta{}_k+a^i{}_\alpha b^\alpha{}_{jk})
\]
holds in $\widetilde{G}_n^2$.
Similarly, we may assume that
\begin{align*}
(F\circ G)^2(x)
&=(f^i(g^m(x)),f^i{}_\alpha(g^m(x))g^\alpha{}_j(x),\\*
&\hphantom{{}={}}\quad f^i{}_{\alpha\beta}(g^m(x))g^\alpha{}_j(x)g^\beta{}_k(x)+f^i{}_\alpha(g^m(x))g^\alpha{}_{jk}(x)).
\end{align*}
Hence we have
\begin{align*}
(ab)^i{}_{jkl}&=a^i{}_{\alpha\beta\gamma}b^\alpha{}_jb^\beta{}_kb^\gamma{}_l+a^i{}_{\alpha\beta}b^\alpha{}_{jl}b^\gamma{}_k+a^i{}_{\alpha\beta}b^\alpha{}_jb^\gamma{}_{kl}+a^i{}_{\alpha\beta}b^\beta{}_lb^\gamma{}_{jk}+a^i{}_\alpha b^\alpha{}_{jkl}
\end{align*}
in $\widetilde{G}_n^3$.
Note that if $F$ is actually derived from a local diffeomorphism, then we have $F(x)=a^i{}_jx^j+\frac12a^i{}_{jk}x^jx^k+\frac16a^i{}_{jkl}x^jx^kx^l+(\text{terms of order greater than $3$})$.

We come back to the bundle $\widetilde{P}^r$ and the canonical form on it.
We have the following

\begin{theorem}[cf.~Theorem~\ref{thmG3}]
\label{thm4.6}
Let $M,N$ be manifolds and $f\colon M\to N$ be a local diffeomorphism.
\begin{enumerate}
\item
We have $f^*\widetilde{P}^r(N)=\widetilde{P}^r(M)$.
\item
If $\theta_M,\theta_N$ denote the canonical forms on $M,N$, then we have $f^*\theta_N=\theta_M$.
\end{enumerate}
\end{theorem}

We can normalize elements of $\widetilde{P}^r(M)$ and $\widetilde{G}_n^r$ as follows (cf.~Theorem~\ref{thm1.13}).
Let $u=j^1_o(F)\in\widetilde{P}^r(M)$.
Let $(U,\varphi)$ be a chart about $\pi^r(u)$ and represent $u$ as $(h^i,h^i{}_{j_1},\ldots,h^i{}_{j_1,\ldots,j_r})$.
We set
\[
\overline{h}^i{}_{j_1,\ldots,j_k}=\frac1{k!}\sum_{\sigma\in\mathfrak{S}_k}h^i{}_{j_{\sigma(1)},\ldots,j_{\sigma(k)}},
\]
where $\overline{h}^i{}_{j_1,\ldots,j_0}$ is understood to be $\overline{h}^i=h^i$, and
\[
\overline{F}^0(x)=\sum_{k=0}^r\overline{h}^i{}_{j_1,\ldots,j_k}x^{j_1}\cdots x^{j_k}.
\]
Actually, we only consider the $r$-jet~of~$\overline{F}^0$ in what follows.

\begin{definition}
We set $\kappa(u)=j^r_o(\overline{F}^0)$.
\end{definition}

\begin{theorem}
\begin{enumerate}
\item
The mapping $\epsilon\colon P^r\to\widetilde{P}^r$ is well-defined and is an embedding of\/ $G_n^r$-bundles.
Moreover, we have $\widetilde{P}^r=\epsilon(P^r)\times_{G_n^r}\widetilde{G}_n^r$.
\item
The mapping $\kappa\colon\widetilde{P}^r\to P^r$ is well-defined bundle morphism as $G_n^r$-bundles.
\item
When regarded as a mapping from $\widetilde{G}_n^r$ to $G_n^r$, $\kappa$ is a homomorphism.
\end{enumerate}
\end{theorem}
\begin{proof}
We only show that the $r$-jet of $\overline{F}^0$ is well-defined.
We set $G(x)=h^i+h^i{}_{j_1}x^{j_1}+\frac12h^i{}_{j_1,j_2}+\cdots+\frac1{r!}h^i{}_{j_1,\ldots,j_r}x^{j_1}\cdots x^{j_r}$.
It is clear that $G$ is well-defined if we ignore terms of order higher than $r$.
On the other hand, we have $G(x)=\overline{F}^0(x)$.
\end{proof}

\begin{remark}
\label{rem4.9}
If we consider a Lie subgroup of $\GL_n(\mathbb{R})$, then we have the same kind of difficulties in the averaging process as in the case of $r=2$.
See also Remarks~\ref{rem1.14} and \ref{rem5.9}.
\end{remark}

We refer to \cite{Saunders} for more about jets.

\section{Torsions of higher order}
We begin with the following

\begin{theorem}[Structural equations~\cite{K_str}, \cite{Bott:Notes}]
\label{thm5.4}
Let $S=\{1,\ldots,k\}$.
Let $S_a=\{s_1,\ldots,s_a\}\subset S$ and $\{t_1,\ldots,t_b\}=S\setminus S_a$, where $S_a=\varnothing$ if $a=0$.
Then, on $P^r$, we have
\[
d\theta^i{}_{j_1,\ldots,j_k}+\sum_{S_a\subset S}\sum_{l=1}^n\theta^i{}_{j_{s_1},\ldots,j_{s_a},l}\wedge\theta^l{}_{j_{t_1},\ldots,j_{t_b}}=0.
\]
These equations are referred as the \textup{structural equations}.
\end{theorem}

Contractions appear in the structural equations.
On $P^r$, the lower indices commute so that we do not need to care about the order of indices.
It is not the case for $\widetilde{P}^r$ so that we should be aware of how we take contractions.

\begin{definition}
\label{def5.2}
Let $k\leq r-1$.
Let $p_0=1,p_1,\ldots,p_k\in\mathbb{N}$ be such that $1\leq p_a\leq a+1$ for $0\leq a\leq k$.
Let $S=(j_1,\ldots,j_k)$ be a ordered tuple of indices.
We set
\[
(\Theta^{k+1(p_1,\ldots,p_k)})^i{}_{j_1,\ldots,j_k}=d\theta^i{}_{j_1,\ldots,j_k}+\sum_{S_a\subset S}\sum_{l=1}^n\theta^i{}_{j_{s_1},\ldots,j_{s_{p_a-1}},\underset{\stackrel{\frown}{p_a}}{l},j_{s_{p_a}},\ldots,j_{s_a}}\wedge\theta^l{}_{j_{t_1},\ldots,j_{t_b}},
\]
where $S_a=(j_{s_1},\ldots,j_{s_a})$ is a ordered subset of $S$ and $(j_{t_1},\ldots,j_{t_b})=S\setminus S_a$ as ordered sets.
If $a=0$, then we set $S_0=\varnothing$.
We call $\Theta^{k(p_1,\ldots,p_{k-1})}$ the \textit{torsions} of order $k$ and of type $(p_1,\ldots,p_{k-1})$.
\end{definition}

\begin{example}
\label{ex5.3}
\begin{enumerate}
\item
The only torsion of order $1$ is $\Theta^1$.
\item
The torsions of order $2$ are $\Theta^{2(1)}$ and $\Theta^{2(2)}$.
We have
\begin{align*}
\Theta^{2(1)}&=d\theta^i{}_j+\theta^i{}_{\alpha j}\wedge\theta^\alpha+\theta^i{}_\alpha\wedge\theta^\alpha_j\\*
&=\Omega^i{}_j+\theta^i{}_{\alpha j}\wedge\theta^\alpha,\\*
\Theta^{2(2)}&=d\theta^i{}_j+\theta^i{}_{j\alpha}\wedge\theta^\alpha+\theta^i{}_\alpha\wedge\theta^\alpha{}_j\\*
&=\Omega^i{}_j+\theta^i{}_{j\alpha}\wedge\theta^\alpha.
\end{align*}
\item
The torsions of order $3$ are $\Theta^{3(1,1)}$, $\Theta^{3(2,1)}$, $\Theta^{3(1,2)}$, $\Theta^{3(2,2)}$, $\Theta^{3(1,3)}$ and $\Theta^{3(2,3)}$.
For example, we have
\begin{align*}
(\Theta^{3(1,3)})^i{}_{j_1j_2}&=d\theta^i{}_{j_1j_2}+\theta^i{}_\alpha\wedge\theta^\alpha{}_{j_1j_2}+\theta^i{}_{\alpha j_1}\wedge\theta^\alpha{}_{j_2}+\theta^i{}_{\alpha j_2}\wedge\theta^\alpha{}_{j_1}+\theta^i{}_{j_1j_2\alpha}\wedge\theta^\alpha,\\*
(\Theta^{3(2,2)})^i{}_{j_1j_2}&=d\theta^i{}_{j_1j_2}+\theta^i{}_\alpha\wedge\theta^\alpha{}_{j_1j_2}+\theta^i{}_{j_1\alpha}\wedge\theta^\alpha{}_{j_2}+\theta^i{}_{j_2\alpha}\wedge\theta^\alpha{}_{j_1}+\theta^i{}_{j_1\alpha j_2}\wedge\theta^\alpha.
\end{align*}
\item
In general, the number of the torsions of order $k$ is equal to $k!$.
\end{enumerate}
\end{example}

We have the following

\begin{theorem}
\label{thm5.7}
We have $P^r=\{u\in\widetilde{P}^r\mid\text{all torsions vanish at $u$}\}$.
\end{theorem}

We need some lemmata for proving Theorem~\ref{thm5.7}.

\begin{lemma}
\label{lem5.8}
Let $a=(a^i{}_{j_1},\ldots,a^i{}_{j_1,\ldots,j_r})$, $b=(b^i{}_{j_1},\ldots,b^i{}_{j_1,\ldots,j_r})\in\widetilde{G}_n^r$.
\begin{enumerate}
\item
If\/ $1\leq k\leq r$, then $(ab)^i{}_{j_1,\ldots,j_k}$ is represented by $a^i{}_{j_1,\ldots,j_l},b^i_{j_1,\ldots,j_l}$ with $l\leq k$.
\item
We have
\begin{align*}
(ab)^i{}_{j_1}&=a^i{}_\alpha b^\alpha_j,\\*
(ab)^i{}_{j_1,\ldots,j_r}&=a^i{}_{\alpha_1,\ldots,\alpha_r}b^{\alpha_1}{}_{j_1}\cdots b^{\alpha_r}{}_{j_r}\\*
&\hphantom{{}={}}+(\text{terms which do not involve $a^i{}_{j_1,\ldots,j_r}$ or $b^i{}_{j_1,\ldots,j_r}$})\\*
&\hphantom{{}={}}+a^i{}_\alpha b^\alpha{}_{j_1,\ldots,j_r}.
\end{align*}
\end{enumerate}
\end{lemma}
\begin{proof}
By the construction, $(ab)^i{}_{j_1,\ldots,j_k}$ is obtained from $(ab)^i{}_{j_1,\ldots,j_{k-1}}$ as follows.
First regard $a,b$ as functions in $x$.
Then, we consider the derivative of $(ab)^i{}_{j_1,\ldots,j_{k-1}}$ with respect to $x$.
After replacing $\pdif{a^i{}_{j_1,\ldots,j_p}}{x^l}$ by $a^i{}_{j_1,\ldots,j_p,\alpha}b^\alpha{}_l$ and $\pdif{b^i{}_{j_1,\ldots,j_p}}{x^l}$ by $b^i{}_{j_1,\ldots,j_p,l}$, we obtain $(ab)^i{}_{j_1,\ldots,j_k}$.
In particular, the number of the lower indices of each terms of $(ab)^i{}_{j_1,\ldots,j_k}$ is equal to $k$.
Hence~1) and the first part of~2) hold.
On the other hand, terms in $(ab)^i{}_{j_1,\ldots,j_r}$ which involve $a^i{}_{j_1,\ldots,j_r}$ appear only if we take the `derivative' of terms in $(ab)^i_{j_1,\ldots,j_{r-1}}$ which involve $a^i{}_{j_1,\ldots,j_{r-1}}$.
In this case, we obtain $a^i{}_{\alpha_1,\ldots,\alpha_r}b^{\alpha_1}{}_{j_1}\cdots b^{\alpha_{r-1}}{}_{j_{r-1}}b^{\alpha_r}{}_{j_r}$.
Similarly, terms in $(ab)^i{}_{j_1,\ldots,j_r}$ which involve $b^i{}_{j_1,\ldots,j_r}$ appear only from the derivative of $a^i{}_\alpha b^\alpha{}_{j_1,\ldots,j_{r-1}}$ and we obtain $a^i{}_\alpha b^\alpha{}_{j_1,\ldots,j_{r-1},j_r}$.
\end{proof}

\begin{lemma}
\label{lem5.9}
Let $\theta=(\theta^0,\theta^1,\ldots,\theta^{r-1})$ be the canonical form on $\widetilde{P}^r$, where $r\geq2$.
In the natural coordinates \textup{(}Notation~\ref{not1.6}\textup{)}, we have
\begin{align*}
(\theta^{r-1})^i{}_{j_1,\ldots,j_{r-1}}&=-v^i{}_\alpha u^\alpha{}_{j_1,j_2,\ldots,j_{r-1}\beta}v^\beta{}_\gamma du^\gamma\\*
&\hphantom{{}={}}%
\quad+(\text{terms which involve only $v^i{}_j$ and $u^i{}_{j_1,\ldots,j_l}$ with $l\leq r-2$}).
\end{align*}
\end{lemma}
\begin{proof}
Let $u\in\widetilde{P}^r$, $X\in T_u\widetilde{P}^r$ and $Y=\theta_u(X)\in T_{\id}\widetilde{P}^{r-1}(\mathbb{R}^n)$.
If we represent $Y$ as $Y=(Y^0,\ldots,Y^{r-1})$ using the natural coordinates, then, we have by Lemma~\ref{lem5.8}~that
\begin{align*}
&\hphantom{{}={}}%
((uY)^{r-1})^i{}_{\alpha_1,\ldots,\alpha_{r-1}}\\*
&=u^i{}_{\alpha_1,\ldots,\alpha_{r-1},l}Y^l\\*
&\hphantom{{}={}}%
+(\textit{terms which do not involve $u^i{}_{j_1,\ldots,j_r}$ or $Y^i{}_{j_1,\ldots,j_{r-1}}$})\\*
&\hphantom{{}={}}%
+u^i{}_lY^l{}_{\alpha_1,\ldots,\alpha_{r-1}},
\end{align*}
where $Y^0=(Y^i)$ and $Y^{r-1}=(Y^i{}_{j_1,\ldots,j_{r-1}})$.
Since $Y^i=\theta^0(X)=v^i{}_\alpha X^\alpha$, we are done (see also~\eqref{eqF10-2} and Lemma~\ref{lem2.6}).
\end{proof}

\begin{proof}[Proof of Theorem~\ref{thm5.7}]
Suppose that torsions of order less than $r$ vanish.
We may assume inductively that the lower indices of $u^i{}_{j_1,\ldots,j_k}$ commute for $k\leq r-1$.
By Lemma~\ref{lem5.9}, the only terms of $\Theta^{r(p_1,\ldots,p_{r-1})}$ which involve $u^i{}_{j_1,\ldots,j_r}$ is derived from $\theta^r\wedge\theta^0$, and we have
\begin{align*}
(\Theta^{r(p_1,\ldots,p_{r-1})})^i{}_{j_1,\ldots,j_{r-1}}&=-v^i{}_\alpha u^\alpha{}_{j_1,\ldots,\underset{\stackrel{\frown}{p_{r-1}}}{l},\ldots,j_{r-1},\beta}v^\beta{}_\gamma du^\gamma\wedge v^l{}_\delta du^\delta\\*
&\hphantom{{}={}}+(\textit{terms which do not involve $u^i{}_{j_1,\ldots,j_r}$}).
\end{align*}
On the other hand, by the structural equations (Theorem~\ref{thm5.4}), $\Theta^{r(p_1,\ldots,p_{r-1})}=0$ if every lower index commutes each other.
This implies that we have
\[
(\Theta^{r(p_1,\ldots,p_{r-1})})^i{}_{j_1,\ldots,j_{r-1}}=-v^i{}_\alpha u^\alpha{}_{j_1,\ldots,\underset{\stackrel{\frown}{p_{r-1}}}{l},\ldots,j_{r-1},\beta}v^\beta{}_\gamma du^\gamma\wedge v^l{}_\delta du^\delta.
\]
Therefore, we have $u^i{}_{j_1,\ldots,\underset{\stackrel{\frown}{p_{r-1}}}{l},\ldots,j_{r-1},\beta}=u^\alpha{}_{j_1,\ldots,\underset{\stackrel{\frown}{p_{r-1}}}{\beta},\ldots,j_{r-1},l}$ if $\Theta^{r(p_1,\ldots,p_{r-1})}=0$.
\end{proof}

\begin{remark}
If torsions of order less than $r$ vanish, then $\Theta^{r(p_1,\ldots,p_{r-1})}$ only depends on $p_{r-1}$.
Indeed, the proof of Theorem~\ref{thm5.7} can be read as a proof of the structural equation in the classical setting where lower indices are commutative.
\end{remark}

\begin{definition}
We set
\begin{align*}
P^\infty&=\varprojlim P^r=\left\{u=(u^r)_{r\in\mathbb{N}}\in\prod_{r\in\mathbb{N}}P^r\;\middle|\;\pi^r_su^r=u^s\ \text{if $s<r$}\right\},\\*
\widetilde{P}^\infty&=\varprojlim\widetilde{P}^r,\\*
G^\infty&=P^\infty(\mathbb{R}^n,o),\\*
\widetilde{G}^\infty&=\widetilde{P}^\infty(\mathbb{R}^n,o),
\end{align*}
and equip them with the limit topology.
To say about $P^\infty$ for example, this is the weakest topology with respect to which the natural mappings $P^\infty\to P^r$ are continuous.
We call $P^\infty$ as the \textit{bundle of frames of infinite order}, and $\widetilde{P}^\infty$ as the \textit{bundle of formal frames of infinite order}, respectively.
If $f\colon(\mathbb{R}^n,o)\to M$ is a local diffeomorphism, then the element of $P^\infty$ determined by $(j^r_o(f))_{r\in\mathbb{N}}$ is represented by $j^\infty_o(f)$.
Similarly, if $F=F^\infty=(F^r)_{r\in\mathbb{N}}$ is an infinite sequence of morphisms from $\widetilde{P}^r(\mathbb{R}^n)$ to $\widetilde{P}^r$ such that $\pi^r_{r-1}\circ F^r=F^{r-1}\circ\pi^r_{r-1}$ and that $F^r(o)=j^1_o(F^{r-1})$, then the element of $\widetilde{P}^\infty$ determined by $(j^1_o(F^r))_{r\in\mathbb{N}}$ is represented by $j^1_o(F)=j^1_o(F^\infty)$.
\end{definition}
Note that $G^\infty$ and $\widetilde{G}^\infty$ are topological groups.

\begin{theorem}
\label{thm5.8}
We have $P^\infty=\{u\in\widetilde{P}^\infty\mid\text{all the torsions vanish at $u$}\}$.
\end{theorem}
\begin{proof}
Suppose that $u\in\widetilde{P}^\infty$ and that all the torsions vanish at $u$.
By taking a chart on $M$, we find an infinite sequence $(u^i,u^i{}_j,\ldots,)$ of tensors of which the lower indices commute.
By \cite{Sternberg}*{Lemma~5}, we can find a local diffeomorphism $f\colon(\mathbb{R}^n,o)\to M$ of class $C^\infty$ which realizes this sequence, namely, we have $j^r_o(f)=(a^i,a^i{}_{j_1},\ldots,a^i{}_{j_1,\ldots,j_r})$ for any $r\in\mathbb{N}$.
If we set $\iota(u)=j^\infty_o(f)$, then $\iota$ gives an isomorphism.
\end{proof}

\begin{remark}
\label{rem5.9}
If we work in the analytic (real or complex) category, then Theorem~\ref{thm5.8} is no longer valid.
Indeed, we need a realization of a given Taylor series as an analytic function, which is in general impossible.
Similarly, if we work with geometric structures, then we can find a smooth function $f$ which realizes a jet at a point, however it is not certain if $f$ preserves the structure.
In order to that, we will need a kind of integrability conditions.
See also Remarks~\ref{rem1.14} and \ref{rem4.9}.
\end{remark}

\begin{example}
\label{ex5.11}
\begin{enumerate}
\item
Let $M$ be a $1$-dimensional M\"obius manifold in the sense that there exists an atlas $\{(U_\lambda,\varphi_\lambda)\}$ of $M$ such that the transition functions are linear fractional transformations which we call M\"obius ones.
In this case formal frames and frames concide and we have $\widetilde{P}^r(M)=P^r(M)$.
We can consider the bundle of formal frames which is defined by using only by M\"obius mappings and their derivatives, which we represent by $\mathcal{M}^r(M)$.
In other words, we consider reductions of $\widetilde{P}^r(M)$ such that the structural groups are derived from $\PGL_1(\mathbb{R})$.
We set now for a function of one variable $f$, $S(f)=\frac{f'''}{f'}-\frac32\left(\frac{f''}{f'}\right)^2$ which is called the \textit{Schwarzian derivative} of $f$.
It is well-known that $f$ is M\"obius if and only if we have $S(f)=0$.
On the other hand, it is also well-known that we have $S(\phi\circ f)=f^*S(\phi)+S(f)$ in general.
Therefore, if we locally set $S=(v^1{}_1u^1{}_{111}-\frac32(v^1{}_1u^1{}_{11})^2)du^1\otimes du^1$, where $v^1{}_1=1/u^1{}_1$, always considering atlases as above, then $S$ is well-defined on $\widetilde{P}^r(M)$.
We see that $j^r(f)\in\widetilde{P}^r(M)=P^r(M)$ is represented by a M\"obuis map from $(\mathbb{R}^1,o)$ to $M$ if and only if $S(f)=0$.
\item
If $n\geq2$, there also is a tensor $\Sigma$ of type $(1,2)$ called the \textit{Schwarzian derivative}, which is still a cocycle such that $\Sigma(f)=0$ if and only if $f$ is projectively linear~\cite{MolMor},~\cite{Oda}.
Hence if we work on projectively flat manifolds, then we can repeat the argument in~1) to obtain a reduction of $\widetilde{P}^r(M)$ and the Schwarzian derivative on it.
We can locally write down the Schwarzian derivative in terms of natural coordinates but we omit it because it is so involved (cf.~\cite{MolMor}*{\S2}).
In this case, we can normalize projective connections in terms of $P^r(M)$.
\item
We can ask if a given anti-symmetric $2$-tensor field is realized as the torsion of a connection.
If we consider projective structures, it is known that there is a connection which preserves the projective structure and of which the torsion is the given tensor field.
On the other hand, if we work on a symplectic manifold of dimension $2n$, where $n\geq2$, then, there exists tensor fields which cannot be realized as the torsion of connections which preserve the symplectic structure~\cite{Weyl},~\cite{KobNag}.
In any case, we need to consider $\widetilde{P}^r(M)$ in order to deal with non-trivial $2$-tensors.
\end{enumerate}
\end{example}

\begin{remark}
The bundle $\widetilde{P}^r(M)$ is quite related with geometric structures as $P^r(M)$ is so.
We will discuss relationship of $\widetilde{P}^2(M)$ and projective structures in~\cite{asuke:2022-2}.
\end{remark}

\section{Infinitesimal deformations of linear connections}
Related to the above constructions, we will discuss infinitesimal deformations of linear connections.
Some of fundamental references are \cite{Rag1}, \cite{Rag2}.

We fix a manifold $M$ and a connection $\nabla$ on $TM$.
Given a chart $(U,\varphi)$, let $\Gamma^i{}_{jk}$ denote the Christoffel symbols of $\nabla$ with respect to $(U,\varphi)$.
That is, if $(x^1,\ldots,x^n)$ denote the standard coordinates for $\varphi(U)$, then
\[
\nabla X=\Gamma^i{}_{jk}\pdif{}{x^i}X^j,
\]
where $\nabla X$ denotes the covariant derivative of $X$ and $X=X^i\pdif{}{x^i}$.
We remark again that the order of the lower indices of the Christoffel symbols are reversed.

Let $\{\nabla_t\}$ be a $1$-parameter smooth family of connections such that $\nabla_0=\nabla$.
Then the Christoffel symbols of $\nabla_t$ are represented as
\[
\Gamma_t{}^i{}_{jk}=\Gamma^i{}_{jk}+\mu^i{}_{jk}(t),
\]
where $\mu^i{}_{jk}(t)\pdif{}{x^i}\otimes dx^k$ is an element of $\Hom(TM,TM)$ for each $t$.

\begin{definition}
We call $\mu^i{}_{jk}(t)$ a \textit{deformation} of $\nabla$.
To be more precise, we call $\mu^i{}_{jk}(t)$ an \textit{actual deformation}.
\end{definition}

If we take the derivative at $t=0$, then we obtain an element of $\Hom(TM,TM)$.
This leads to the following

\begin{definition}
An \textit{infinitesimal deformation} of $\nabla$ is an element of $\Hom(TM,TM)$.
\end{definition}

It is clear that if $\mu^i{}_{jk}(t)$ is an actual deformation, then $\pdif{\mu^i{}_{jk}}{t}(0)$ is an infinitesimal deformation.

If $\mu$ is an infinitesimal deformation of $\nabla$, then $\mu$ is represented by a family $\{\mu^i{}_{jk}\}$ with respect to an atlas.
If $(U,\varphi)$, $(U,\widehat{\varphi})$ are charts and if $\phi$ denotes the transition function, then we have
\[
\widehat{\mu}^i{}_{\alpha\beta}(D\phi)^\alpha{}_j(D\phi)^\beta{}_k=(D\phi)^i{}_{\alpha}\mu^\alpha{}_{jk}.
\]

In what follows, we will show that a pair of a connection and its infinitesimal deformation can be understood as a connection on a certain bundle.
For this purpose, we will recall some basic notions.
Let $G$ be a Lie group and $TG$ the tangent bundle of $G$, which is also a Lie group called the tangent Lie group of $G$.
We consider the adjoint action of $G$ on $\mathfrak{g}$, where $\mathfrak{g}$ is regarded as the vector space of left invariant vector fields on $G$.
Then, it is well-known that $TG$ is isomorphic to $G\ltimes\mathfrak{g}$.
The product in $G\ltimes\mathfrak{g}$ is given by
\[
(A,X).(B,Y)=(AB,\Ad_{B^{-1}}X+Y).
\]
If $G$ is linear and if $G\subset\GL_n(\mathbb{R})$, then, we have a matrix representation of $TG$ to $\GL_{2n}(\mathbb{R})$ given by
\[
(A,X)\mapsto\begin{pmatrix}
A \\
AX & A
\end{pmatrix}.
\]

The Lie algebra of $G\ltimes\mathfrak{g}$ is described as follows.
We omit the proof.

\begin{lemma}
\label{lem6.3}
Let $\mathfrak{h}$ be the Lie algebra of $G\ltimes\mathfrak{g}$.
\begin{enumerate}
\item
The Lie algebra $\mathfrak{h}$ is a vector space $\mathfrak{g}\times\mathfrak{g}$ equipped with the bracket
\[
[(\dot{A},\dot{X}),(\dot{B},\dot{Y})]=([\dot{A},\dot{B}],\ad_{\dot{X}}\dot{B}+\ad_{\dot{A}}\dot{Y}).
\]
\item
The adjoint action of $G\ltimes\mathfrak{g}$ on $\mathfrak{h}$ is given by
\[
\Ad_{(A,X)}(\dot{B},\dot{Y})=(\Ad_A\dot{B},\Ad_A(\ad_X\dot{B}+\dot{Y})).
\]
\item
Suppose that $G$ is linear.
Then, an element $(\dot{A},\dot{B})\in\mathfrak{h}$ corresponds to $\begin{pmatrix}
\dot{A}\\
\dot{B} & \dot{A}
\end{pmatrix}\in\mathfrak{tg}$ under the matrix representation.
\end{enumerate}
\end{lemma}

Let $P$ be a principal $G$-bundle over $M$, and set $\underline{\mathfrak{g}}=M\times\mathfrak{g}$.
The projection from $\underline{\mathfrak{g}}$ to $M$ is represented by $\nu$.

\begin{definition}
Let $P\ltimes\underline{\mathfrak{g}}$ with the projection $\varpi$ to $M$ be a principal $TG$-bundle defined as follows.
We set $P\ltimes\underline{\mathfrak{g}}=\{(u,X)\in P\times\underline{\mathfrak{g}}\mid\pi(u)=\nu(X)\}$ and $\varpi(u,X)=\pi(u)$, namely, we consider the fiber product.
If $(B,Y)\in TG$, then we~set
\[
(u,X).(B,Y)=(u.B,\Ad_{B^{-1}}X+Y).
\]
\end{definition}

\begin{theorem}
\label{thm6.16}
There is a one-to-one correspondence between the following\textup{:}
\begin{enumerate}
\item
Pairs of connections on $TM$ and their infinitesimal deformations.
\item
Connections on $P^1(M)\ltimes\underline{\mathfrak{gl}_n}(\mathbb{R})$.
\item
Sections from $P^1(M)\ltimes\underline{\mathfrak{gl}_n}(\mathbb{R})$ to $\widetilde{P}^2(M)\ltimes\underline{\widetilde{\mathfrak{g}}^2_n}$ equivariant under the $\GL_n(\mathbb{R})\ltimes\mathfrak{gl}_n(\mathbb{R})$-action.
\end{enumerate}
\end{theorem}
\begin{proof}
We set $P=P^1(M)\ltimes\underline{\mathfrak{gl}_n}(\mathbb{R})$, $G=\GL_n(\mathbb{R})\ltimes\gl_n(\mathbb{R})$ and let $\mathfrak{g}$ denote the Lie algebra of $G$.
First, the transition functions for $P$ is given as follows.
Let $(U,\varphi)$ and $(\widehat{U},\widehat{\varphi})$ be charts on $M$ and $\phi$ the transition function.
Then, $P$ is trivial on $U$.
Let $((x,g);X)=((x^i,g^i{}_j);X^i{}_j)\in\varphi(U)\times\GL_n(\mathbb{R})\times\mathfrak{gl}_n(\mathbb{R})$ be coordinates for $P|_U$ and $((\widehat{x},\widehat{g});\widehat{X})$ for $P|_{\widehat{U}}$.
We have $((\widehat{x},\widehat{g});\widehat{X})=((\phi(x),D\phi(x)g);X)$.
As $(D\phi(x)g,X)=(D\phi(x),0)(g,X)$ in $G=\GL_n(\mathbb{R})\ltimes\gl_n(\mathbb{R})$, the transition function for $P$ is given by $(D\phi(x),0)$.
Let now $\{\omega\}$, where $\omega=(\theta,\mu)$, be a family of $\mathfrak{g}$-valued $1$-form on $M$.
This family represents a connection on $P$ if and only if we~have
\[
(\omega,\theta)=(D\phi,0)^{-1}d(D\phi,0)+\Ad_{(D\phi,0)^{-1}}(\widehat{\omega},\widehat{\theta}).
\]
This condition is equivalent to
\begin{align*}
\theta&=D\phi(x)^{-1}dD\phi+\Ad_{D\phi^{-1}}\widehat{\theta},\\*
\mu&=\Ad_{D\phi^{-1}}\widehat{\mu}.
\end{align*}
Hence $(\theta,\mu)$ represents a connection on $P$ if and only if $\mu$ is an infinitesimal deformation of $\theta$.
It is clear that the correspondence is one-to-one.
Next, we show that the conditions 2) and 3) are equivalent.
By theorems of Garc\'\i a (Theorems~2 and 3 of~\cite{Garcia}, cf. Theorem~\ref{thmG11}), connections on $P^1(M)\ltimes\underline{\gl_n}(\mathbb{R})$ are in a one-to-one correspondence between $G$-equivariant sections of $\mathcal{J}(P^1(M)\ltimes\underline{\gl_n}(\mathbb{R}))\to P^1(M)\ltimes\underline{\gl_n}(\mathbb{R})$.
Hence it suffices to show that $\mathcal{J}(P^1(M)\ltimes\underline{\gl_n}(\mathbb{R}))$ and $\widetilde{P}^2(M)\ltimes\underline{\widetilde{\mathfrak{g}}^2_n}$ are isomorphic as $G$-bundles.
Let $j^1_p(s)\in\mathcal{J}(P^1(M)\ltimes\underline{\gl_n}(\mathbb{R}))$, where $s$ is a section of $P^1(M)\ltimes\underline{\gl_n}(\mathbb{R})\to M$ about $p$.
Let $(U,\varphi)$ be a chart such that $\varphi(p)=o$ and that $D\varphi^{-1}(o)=s(p)$.
We represent $s\circ\varphi=(s_1,s_2)$ and associate $j^1_p(s)$ with $(j^1_o(\varphi^{-1},s_1),j^1_o(s_2))$ (cf. Lemma~\ref{lem3.5}).
This gives a desired isomorphism.
Indeed, the correspondence is locally given as follows.
Let $(U,\varphi)$ be a chart and $(x^i,(a^i{}_j,b^i{}_j),(a^i{}_{jk},b^i{}_{jk}))$ be the representation of $s$ with respect to the Garc\'\i a coordinates.
We adopt as the coordinates for $\widetilde{P}^2\ltimes\underline{\widetilde{\mathfrak{g}}^2_n}$ the product of the natural coordinates for $\widetilde{P}^2$ and the trivial one for $\underline{\widetilde{\mathfrak{g}}^2_n}$.
Then, $j^1_p(s)$ is mapped to $(x^i,(a^i{}_j,a^i{}_{j\alpha}a^\alpha{}_k),(b^i{}_j,b^i{}_{jk}))$.
If $(g,X)\in\GL_n(\mathbb{R})\ltimes\gl_n(\mathbb{R})$, then we have
\begin{align*}
&\hphantom{{}={}}%
(x^i,(a^i{}_j,b^i{}_j),(a^i{}_{jk},b^i{}_{jk})).(g,X)\\*
&=(x^i,(a^i{}_\alpha g^\alpha{}_j,h^i{}_\alpha b^\alpha{}_\beta g^\beta{}_j+X^i{}_j),(a^i{}_{\alpha j}g^\alpha{}_k,h^i{}_\alpha b^\alpha{}_{\beta k}g^\beta{}_j)),
\end{align*}
where $(h^i{}_j)=(g^i{}_j)^{-1}$.
On the other hand, we have
\begin{align*}
&\hphantom{{}={}}%
(x^i,(a^i{}_j,a^i{}_{j\alpha}a^\alpha{}_k),(b^i{}_j,b^i{}_{jk})).(g,X)\\*
&=(x^i,(a^i{}_\alpha g^\alpha{}_j,a^i{}_{\beta\alpha}g^\beta{}_ja^\alpha{}_\beta g^\beta{}_k),(h^i{}_\alpha b^\alpha{}_\beta g^\beta{}_j+X^i{}_j,h^i{}_\alpha b^\alpha{}_{\beta k}g^\beta{}_j)).
\end{align*}
Hence we obtained a morphism from $\mathcal{J}(P^1(M)\ltimes\underline{\gl_n}(\mathbb{R}))$ to $\widetilde{P}^2(M)\ltimes\underline{\widetilde{\mathfrak{g}}^2_n}$ equivariant under the $\GL_n(\mathbb{R})\ltimes\gl_n(\mathbb{R})$-actions.
It is easy to see that this morphism is indeed an isomorphism.
\end{proof}

\begin{remark}
The canonical form on $\mathcal{J}(P^1(M)\ltimes\underline{\gl_n}(\mathbb{R}))$, which corresponds to the canonical form of order $2$ on $\widetilde{P}^2(M)\ltimes\underline{\widetilde{\mathfrak{g}}^2_n}$, is locally given~by
\[
(c^i{}_\alpha(da^\alpha{}_j-a^\alpha{}_{j\beta}dx^\beta),db^i{}_j-b^i{}_{j\alpha}dx^\alpha-b^i{}_\alpha c^\alpha{}_\beta(da^\beta{}_j-a^\beta{}_{j\gamma}dx^\gamma)+c^i{}_\alpha(da^\alpha{}_\beta-a^\alpha{}_{\beta\gamma}dx^\gamma)b^\gamma{}_j),
\]
where $(c^i{}_j)=(a^i{}_j)^{-1}$, with respect to the Garc\'\i a coordinates.
\end{remark}

An explanation of Theorem~\ref{thm6.16} can be given by using an auxiliary structure.
For this purpose, we recall $2$-tangent bundles (see~\cite{IY} for details).

\begin{definition}
We set $T^2M=T(TM)$ and call $T^2M$ as the \textit{$2$-tangent bundle}.
The projection from $T^2M$ to $TM$ is represented by $p^2$.
\end{definition}

Charts and transition functions on $T^2M$ are given as follows.
Let $(U,\varphi)$, $(\widehat{U},\widehat{\varphi})$ be charts of $M$ and $\phi$ the transition function from $U$ to $\widehat{U}$.
Then, $TM$ is trivial on $U$ and $\widehat{U}$.
If $(x,v)$ denote local coordinates for $TM$ on $p^{-1}(U)$, where $p\colon TM\to M$ is the projection, then the transition function is given as $(x,v)\mapsto(\phi(x),D\phi(x)v)$.
Further, $T^2M$ is trivial on $p^{-1}(U)$ and $p^{-1}(\widehat{U})$.
If $(x,v;\dot{x},\dot{v})$ denote local coordinates for $T^2M$ on $(p^2)^{-1}(p^{-1}(U))$, then the transition function is given as follows.
Let $\gamma\colon(-\epsilon,\epsilon)\to TM$ be a curve.
We represent this curve as $(x(t),v(t))$ using a chart.
We have $(\widehat{x}(t),\widehat{v}(t))=(\phi(x(t)),D\phi(x(t))v(t))$ so that
\stepcounter{theorem}
\[
(\widehat{x},\widehat{v};\dot{\widehat{x}},\dot{\widehat{v}})%
=(\phi(x),D\phi(x)v;D\phi(x)\dot{x},H\phi(x)v\dot{x}+D\phi(x)\dot{v}),
\tag{\thetheorem}
\label{eq6.5}
\]
where $H\phi=D(D\phi)$ and $H\phi(x)v\dot{x}=(H\phi)^i{}_{\alpha\beta}v^\alpha\dot{x}^\beta$.
We refer to these coordinates as coordinates \textit{induced} by $(U,\varphi)$.

\begin{definition}
Let $(U,\varphi)$ be a chart on $M$ and consider induced coordinates.
If $u\in TM$, then we set
\[
{\left(\pdif{}{x^i}\right)^V}_u=\pdif{}{v^i}_u
\]
and call it the \textit{vertical lift} of $\pdif{}{x^i}_x$, where $x=p(u)$.
In general, we extend the vertical lift by linearity.
We set
\[
V=\{\text{vectors on $TM$ which are the vertical lifts of vectors on $M$}\}.
\]
\end{definition}

The following is easy.

\begin{lemma}
We have $T^2M/V\cong\pi^*TM$ as vector bundles over $TM$.
\end{lemma}

\begin{definition}
Let $\nabla$ be a connection on $TM$.
Let $(U,\varphi)$ be a chart on $M$ and consider induced coordinates.
Let $\Gamma^i{}_{jk}$ be the Christoffel symbols of $\nabla$ with respect to $(U,\varphi)$.
If $u\in TM$, then we set
\stepcounter{theorem}
\[
{\left(\pdif{}{x^i}\right)^H}_u=\pdif{}{x^i}_u-\Gamma(x)^\alpha{}_{i\beta}v^\beta\pdif{}{v^\alpha}_u
\tag{\thetheorem}
\label{eq6.6}
\]
and call it the \textit{horizontal lift} of $\pdif{}{x^i}_x$, where $x=p(u)$.
In general, we extend the horizontal lift by linearity.
Finally, we set
\[
H=\{\text{vectors on $TM$ which are the horizontal lifts of vectors on $M$}\}.
\]
\end{definition}

Note that $\pdif{}{x^i}$ on the left hand side of \eqref{eq6.6} refers to a vector on $M$, on the other hand, the same symbol on the right hand side refers to a vector on $TM$, we always considering charts and induced coordinates.

\begin{proposition}
The pairs $(V,p^2|_V)$ and $(H,p^2|_H)$ are isomorphic to $\pi^*TM$.
\end{proposition}
\begin{proof}
By~\eqref{eq6.5}, we have
\[
\left(\pdif{}{\widehat{x}^i},\pdif{}{\widehat{v}^i}\right)\begin{pmatrix}
D\phi(x)\\
H\phi(x)v & D\phi(x)
\end{pmatrix}=\left(\pdif{}{x^i},\pdif{}{v^i}\right).
\]
Therefore, we have
\begin{align*}
&\hphantom{{}={}}%
\left(\left(\pdif{}{\widehat{x}^i}\right)^H,\left(\pdif{}{\widehat{x}^i}\right)^V\right)\begin{pmatrix}
I_n\\
\widehat{\Gamma}\widehat{v} & I_n
\end{pmatrix}\begin{pmatrix}
D\phi(x)\\
H\phi(x)v & D\phi(x)
\end{pmatrix}\begin{pmatrix}
I_n\\
-\Gamma v & I_n
\end{pmatrix}\\*
&=\left(\left(\pdif{}{x^i}\right)^H,\left(\pdif{}{x^i}\right)^V\right).
\end{align*}
By~\eqref{eq2.14}, we have
\[
\begin{pmatrix}
I_n\\
\widehat{\Gamma}\widehat{v} & I_n
\end{pmatrix}\begin{pmatrix}
D\phi(x)\\
H\phi(x)v & D\phi(x)
\end{pmatrix}\begin{pmatrix}
I_n\\
-\Gamma v & I_n
\end{pmatrix}=\begin{pmatrix}
D\phi(x)\\
& D\phi(x)
\end{pmatrix}.
\]
Hence $(V,p^2|_V)$ and $(H,p^2|_H)$ are isomorphic to $\pi^*TM$.
\end{proof}

\begin{remark}
It is known that the horizontal lifts recovers connections~\cite{IY} if connections are torsion free.
It remains valid in our setting as follows.
Let $X$ be a vector field on $M$ and $v\in T_pM$.
We regard $X$ as a mapping from $M$ to $TM$ and let $DX\colon TM\to T^2M$ be the derivative.
We represent $X=f^i\pdif{}{x^i}$ and $v=v^i\pdif{}{x^i}_p$ on a chart.
Then, we have
\[
\nabla_vX(p)=\left(\pdif{f^\alpha}{x^\beta}(p)v^\beta+\Gamma^\alpha{}_{\beta\gamma}(p)f^\beta(p)v^\gamma\right)\pdif{}{x^\alpha}_p.
\]
On the other hand, we have
\begin{align*}
DX(v)_p&=v^\alpha\pdif{}{x^\alpha}_{X(p)}+\pdif{f^\alpha}{x^\beta}(p)v^\beta\pdif{}{v^\alpha}_{X(p)},\\*
v^H_{X(p)}&=v^\alpha\pdif{}{x^\alpha}_{X(p)}-\Gamma^\alpha{}_{\beta\gamma}(p)v^\beta f^\gamma(p)\pdif{}{v^\alpha}_{X(p)}.
\end{align*}
It follows that we have
\[
(\nabla_vX(p)+T(v,X(p)))^V_{X(p)}=DX(v)_p-v^H_{X(p)},
\]
where $T$ denotes the torsion of $\nabla$.
\end{remark}

It is clear that $T^2M=H\oplus V$.
This identification can be seen as an isomorphism between $T^2M$ and $p^{2*}(TM)$.
Note that the horizontal lifts give sections of the projection from $T^2M$ to $TM$.
If $\nabla$ is a connection on $TM$ and $\mu$ its infinitesimal deformation, then we can define a connection, say $\widetilde\nabla$, on $H\oplus V$.
Indeed, let $\theta^i{}_j$, $\mu^i{}_j$ be matrix representations of $\nabla$ and $\mu$ with respect to $\pdif{}{x^i}$ on $M$.
On the other hand, we consider $\left(\left(\pdif{}{x^i}\right)^H,\left(\pdif{}{x^i}\right)^V\right)$ as a local frame for $T^2M=H\oplus V$.
Then locally defined $\tgl_n(\mathbb{R})$-valued $1$-forms $\begin{pmatrix}
\theta^i{}_j \\
\mu^i{}_j & \theta^i{}_j
\end{pmatrix}$
give rise to a connection $\widetilde{\nabla}$.
The diagonal part correspond to the pull-back connection of $\nabla$ on $V$ and $H$ identified with $p^*(TM)$.
On $H$, we have an additional term valued in $V$ which is given by an element of $\Hom(H,V)$.
This corresponds to the pull-back of $\mu\in\Hom(TM,TM)$.
The connection $\widetilde\nabla$ is related with the one obtained by Theorem~\ref{thm6.16} as follows.
Let $\mathcal{P}^1(TM)$ be the principal bundle associated with $T^2M$, which is a $T\GL_n(\mathbb{R})$-bundle over $TM$ described as follows.
Let $(U,\varphi)$ be a chart of $M$ and $(x,v)$ be the associated coordinates for $TM|_U$.
Let $\{f,g\}\colon(\mathbb{R}^n\times\mathbb{R}^n,(o,o))\to TM$ be a mapping of the~form
\[
(\xi,\eta)\mapsto\{f,g\}(\xi,\eta)=(f(\xi),g(\xi)+Df(\xi)\eta)
\]
with respect to a chart.
If we ignore terms of order greater than $1$ with respect to $\eta$, this property is independent of charts.
Note also that such a mapping is a local diffeomorphism if and only if $Df(o)\in\GL_n(\mathbb{R})$.

\begin{definition}
We set
\begin{align*}
\mathcal{T}(M)&=\{\text{$\{f,g\}$ as above}\},\\*
\mathcal{T}_0(M)&=\{\{f,g\}\in\mathcal{T}\mid g(o)=o\},\\*
\mathcal{P}'(TM)&=\{j^1_{(o,o)}(\{f,g\})\mid\{f,g\}\in\mathcal{T}(M)\},\ \text{and}\\*
\mathcal{G}&=\{j^1_{(o,o)}(\{f,g\})\mid\{f,g\}\in\mathcal{T}(\mathbb{R}^n),\ f(o)=g(o)=o\},
\end{align*}
where $T\mathbb{R}^n$ is naturally trivialized.
\end{definition}

\begin{lemma}
\label{lem6.13}
\begin{enumerate}
\item
The set $\mathcal{G}$ is a group of which the product is the composition, and is isomorphic to $T\GL_n(\mathbb{R})$.
\item
The bundle $\mathcal{P}'(TM)$ is isomorphic to $\mathcal{P}^1(TM)$ as $T\GL_n(\mathbb{R}^n)$-bundles.
\end{enumerate}
\end{lemma}
\begin{proof}
We have $D_{(o,o)}(\{f,g\})=\begin{pmatrix}
Df(o)\\
Dg(o) & Df(o)
\end{pmatrix}$.
On the other hand, we have
\stepcounter{theorem}
\begin{align*}
\tag{\thetheorem}
\label{eq6.14}
\{f,g\}.\{f',g'\}(\xi,\eta)&=\{f,g\}(f'(\xi),g'(\xi)+Df'(\xi)\eta)\\*
&=(f\circ f'(\xi),g\circ f'(\xi)+Df(f'(\xi))(g'(\xi)+Df'(\xi)\eta))\\*
&=(f\circ f'(\xi),h(\xi)+D(f\circ f')(\xi)\eta),
\end{align*}
where $h(\xi)=g\circ f'(\xi)+Df(f'(\xi))g'(\xi)$.
Hence we have
\begin{align*}
&\hphantom{{}={}}%
D_{(o,o)}(\{f,g\}.\{f',g'\})\\*
&=\begin{pmatrix}
Df(f'(o))Df'(o)\\
Dg(f'(o))Df'(o)+Hf(f'(o))Df'(o)g'(o)+Df(f'(o))Dg'(o) & Df(f'(o))Df'(o)
\end{pmatrix}
\end{align*}
If $f'(o)=o$ and if $g'(o)=o$, then we have
\begin{align*}
&\hphantom{{}={}}%
D_{(o,o)}(\{f,g\}.\{f',g'\})\\*
&=\begin{pmatrix}
Df(o)Df'(o)\\
Dg(o)Df'(o)+Df(o)Dg'(o) & Df(o)Df'(o)
\end{pmatrix}\\*
&=\begin{pmatrix}
Df(o)\\
Dg(o) & Df(o)
\end{pmatrix}\begin{pmatrix}
Df'(o)\\
Dg'(o) & Df'(o)
\end{pmatrix}.
\end{align*}
Similarly, if $(U,\varphi)$, $(\widehat{U},\widehat{\varphi})$ be charts and $\phi$ the transition function, then we have
\[
D\phi\circ\{f,g\}(\xi,\eta)=(\phi\circ f(\xi),D\phi(f(\xi))(g(\xi)+Df(\xi)\eta)).
\]
Hence we have
\[
D\phi(j^1_{(o,o)}(\{f,g\}))=\begin{pmatrix}
D\phi(f(o)) \\
H\phi(f(o))g(o) & D\phi(f(o))
\end{pmatrix}\begin{pmatrix}
Df(o)\\
Dg(o) & Df(o)
\end{pmatrix}.
\]
Therefore $\mathcal{P}'(TM)$ and $\mathcal{P}^1(TM)$ are isomorphic as $T\GL_n(\mathbb{R})$-bundles.
\end{proof}

We set
\begin{align*}
\mathcal{P}^1_0(TM)&=\{j^1_{(o,o)}(\{f,g\})\mid\{f,g\}\in\mathcal{T}_0\}.
\end{align*}

The bundle $\mathcal{P}^1_0(TM)$ is a principal $T\GL_n(\mathbb{R})$-bundle which is the restriction of $\mathcal{P}^1(TM)$ to the zero section of $TM\to M$.
Recall that we have an isomorphism between $T\GL_n(\mathbb{R})$ and $\GL_n(\mathbb{R})\ltimes\mathfrak{gl}_n(\mathbb{R})$.
We have the following

\begin{lemma}
\label{lem6.18}
The bundles $\mathcal{P}^1_0(TM)$ and $P^1(M)\ltimes\underline{\mathfrak{gl}_n}(\mathbb{R})$ are isomorphic as principal $\GL_n(\mathbb{R})\ltimes\mathfrak{gl}_n(\mathbb{R})$-bundles.
\end{lemma}
\begin{proof}
By the proof of Lemma~\ref{lem6.13}, we see that the transition function on $\mathcal{P}^1_0(TM)$ is given by $\begin{pmatrix}
D\phi(f(o))\\
& D\phi(f(o))
\end{pmatrix}$ because we have $g(o)=o$.
If we associate $j^1_{(o,o)}(\{f,g\})$ with $((f(o),Df(o)),Df(o)^{-1}Dg(o))$, then obtain the desired isomorphism.
\end{proof}

Let now consider a pair of a connection on $TM$ and its infinitesimal deformation.
This induces a connection on $\mathcal{P}^1(TM)$.
By restricting this latter connection to $\mathcal{P}^1_0(TM)$, we obtain a connection on $P^1(M)\ltimes\underline{\mathfrak{gl}_n}(\mathbb{R})$ by Lemma~\ref{lem6.18}.
By Lemma~\ref{lem6.3}, we see that this is the connection given by Theorem~\ref{thm6.16} 

\begin{remark}
There is a description of $\widetilde{P}^2(M)\ltimes\underline{\widetilde{\mathfrak{g}}^2_n}$ similar to that of $P^1(M)\ltimes\underline{\gl_n}(\mathbb{R})$.
We consider mappings from $T^2\mathbb{R}^n$ to $T^2M$ locally defined by
\begin{align*}
&\hphantom{{}={}}%
\{f,g;F,G\}(\xi,\eta;\dot{\xi},\dot{\eta})\\*
&=(f(\xi),Df(\xi)\eta+g(\xi);F(\xi)\dot{\xi},(DF(\xi)^i{}_{\alpha j}\eta^\alpha+G(\xi))\dot{\xi}+F(\xi)\dot{\eta}),
\end{align*}
where we assume that $Df(o)=F(o)$, $Dg(o)=G(o)$ and that $Df(o)$ is regular as a matrix.
We set
\begin{align*}
\mathcal{T}^2(M)&=\{\text{$\{f,g;F,G\}$ as above}\},\\*
\mathcal{T}^2_0(M)&=\{\{f,g;F,G\}\in\mathcal{T}^2(M)\mid g(o)=o\},\\*
\mathcal{P}^2(TM)&=\{j^1_{(o,o)}(\{f,g;F,G\})\mid\{f,g;F,G\}\in\mathcal{T}^2(M)\},\\*
\mathcal{P}^2_0(TM)&=\{j^1_{(o,o)}(\{f,g;F,G\})\mid\{f,g;F,G\}\in\mathcal{T}^2_0(M)\},\\*
\mathcal{G}^2&=\{j^1_{(o,o)}(\{f,g;F,G\})\mid\{f,g;F,G\}\in\mathcal{T}^2(M),\ f(o)=g(o)=o\}
\end{align*}
We can show that $\mathcal{G}^2$ is a group isomorphic to $\widetilde{G}_n^2\ltimes\widetilde{\mathfrak{g}}^2_n$ of which $\GL_n(\mathbb{R})\ltimes\mathfrak{gl}_n(\mathbb{R})$ is a subgroup.
We can also show the following
\begin{lemma}
The bundle $\mathcal{P}^2_0(TM)$ is isomorphic to $\widetilde{P}^2(M)\ltimes\underline{\widetilde{\mathfrak{g}}^2_n}$ as a $\widetilde{G}^2_n\ltimes\widetilde{\mathfrak{g}^2}$-bundle.
\end{lemma}
We omit the proof because they are parallel to the case of $\mathcal{P}^1_0(TM)$.
\end{remark}

\section{Application to deformations of foliations}
\label{sec6}
We consider regular (non-singular) foliations.
Associated with such foliations, there are natural connections called \textit{Bott connections}.
If foliations are deformed, Bott connections are also deformed according to deformations.
We will discuss infinitesimal deformations of Bott connections associated with infinitesimal deformations of foliations.
Let $\CF$ be a foliation of $M$, of codimension $q$.
Let $\{(U_\lambda,\varphi_\lambda)\}$ be a foliation atlas, that is, we have homeomorphisms $U_\lambda\cong V_\lambda\times T_\lambda$ such that the restriction of $\CF$ to $U_\lambda$ is given by $\{V_\lambda\times\{y\}\}_{y\in T_\lambda}$, where $V_\lambda$ and $T_\lambda$ are balls in $\mathbb{R}^{\dim M-q}$ and $\mathbb{R}^q$, respectively.
Let $p_\lambda$ denote the projection from $U_\lambda$ to $T_\lambda$.
Then, the transition function from $U_\lambda$ to $U_\mu$ is of the form $(x_\lambda,y_\lambda)\mapsto(\psi_{\mu\lambda}(x_\lambda,y_\lambda),\gamma_{\mu\lambda}(y_\lambda))$.
We refer to $\gamma_{\mu\lambda}$ as the \textit{holonomy map}.
Let $T\CF$ be the tangent bundle of $\CF$, which is the subbundle of $TM$ which consists of vectors tangent to leaves of $\CF$.
The bundle $T\CF$ locally consists of vectors tangent to $V_\lambda\times\{y\}$.

\begin{definition}
The quotient bundle $Q(\CF)=TM/T\CF$ is called as the \textit{normal bundle} of $\CF$.
\end{definition}

\begin{definition}
A connection on $Q(\CF)$ is called a \textit{Bott connection} if and only if $\nabla_XY=\mathcal{L}_XY$ holds for $X\in T\CF$, where $\mathcal{L}_X$ denotes the Lie derivative with respect to $X$.
\end{definition}
It is well-known that Bott connections always exist in the $C^\infty$-category.
Bott connections enjoy the following property.

\begin{lemma}
\label{lem7.3}
Let $\nabla$ be a Bott connection on $Q(\CF)$.
Let $(U,\varphi)$ be a foliation chart and $(x,y)$ be coordinates for $U\cong V\times T$.
If\/ $\theta^i{}_j$ denote the components of connection matrix with respect to $\pdif{}{y^1},\ldots,\pdif{}{y^q}$, then $\theta^i{}_j$ do not involve $dx^k$'s.
\end{lemma}
\begin{proof}
This is because we have $\nabla_{\pdif{}{x^k}}\pdif{}{y^j}=\mathcal{L}_{\pdif{}{x^k}}\pdif{}{y^j}=0$.
\end{proof}

\begin{remark}
Lemma~\ref{lem7.3} does \textit{not} mean that $\theta^i{}_j$ are independent of $x^k$.
\end{remark}

In the setting of Lemma~\ref{lem7.3}, we can represent $\theta^i{}_j$ as $\theta^i{}_j=\Gamma^i{}_{jk}dy^k$.
The functions $\Gamma^i{}_{jk}$ are referred as the \textit{Christoffel symbols}.
Note that the order of lower indices are always reversed as in the previous sections.

\begin{proposition}
Let $\nabla$ a connection on $Q(\CF)$.
Let $e=(e_i)$ be a local trivialization of $Q(\CF)$ which is \textup{foliated} or locally projectable in the sense that each $e_i$ is of the form $f^\alpha{}_i\pdif{}{y^\alpha}$ with $f^\alpha{}_i$ being functions on $y$ independent of $x$, where $(x,y)$ are local coordinates on a foliation chart.
Let $(\theta^i{}_j)$ be the connection matrix of\/ $\nabla$ with respect to $e$.
Then, $\nabla$ is a Bott connection if and only if\/ $\theta^i{}_j|_{T\CF}=0$.
\end{proposition}
\begin{proof}
Let $\Gamma^i{}_j$ be the connection matrix of $\nabla$ with respect to $\left(\pdif{}{y^i}\right)$.
If we set $F=(f^i{}_j)$, then we have $\theta^i{}_j=(F^{-1})^i{}_\alpha dF^\alpha{}_j+(F^{-1})^i{}_\alpha\Gamma^\alpha{}_\beta F^\beta{}_j$.
As $e$ is foliated, $\theta^i{}_j$ do not involve $dx^i$ if and only if so do not $\Gamma^i{}_j$.
\end{proof}

By Theorem~\ref{thm4.6} we can form $\widetilde{G}^r_q$-bundles $\widetilde{P}^r(\CF)$ by pasting $p_\lambda{}^*\widetilde{P}^r(T_\lambda)$.
To be precise, suppose that $U_\lambda\cap U_\mu\neq\varnothing$, and let $u_\lambda\in p_\lambda^*\widetilde{P}^r(T_\lambda)$ and $u_\mu\in p_\mu^*\widetilde{P}^r(T_\mu)$.
We have naturally have $p_\lambda^*\widetilde{P}^r(T_\lambda)\cong V_\lambda\times T_\lambda\times\widetilde{G}^r_q$.
Let we represent $u_\lambda=(x_\lambda,y_\lambda,g_\lambda)$ and $u_\mu=(x_\mu,y_\mu,g_\mu)$.
Then, $\gamma_{\mu\lambda}$ gives a locally defined map from $\widetilde{P}^r(T_\lambda)$ to $\widetilde{P}^r(T_\mu)$, which we represent by $\gamma_{\mu\lambda*}$.

\begin{definition}
We say that $u_\lambda\sim u_\mu$ if and only if $x_\lambda=x_\mu$ and $(y_\mu,g_\mu)=\gamma_{\mu\lambda*}(y_\lambda,g_\lambda)$.
We set $\widetilde{P}^r(\CF)=\left(\bigsqcup_{\lambda}p_\lambda^*\widetilde{P}^r(T_\lambda)\right)/\sim$.
The natural projection from $\widetilde{P}^r(\CF)$ to $M$ is represented by $\pi^r$.
\end{definition}

It is easy to see that $\widetilde{P}^r(\CF)$ is independent of the choice of foliation atlases.
It is clear that $\widetilde{P}^1(\CF)$ is isomorphic to the frame bundle of $Q(\CF)$ which is represented by~$P^1(\CF)$.

There are natural foliations of $\widetilde{P}^r(\CF)$.
Indeed, if $U\cong V\times T$ is a foliation chart, then $\widetilde{P}^r(\CF)$ is trivial, namely, we have $\widetilde{P}^r(\CF)|_U\cong U\times\widetilde{G}_q^r\cong V\times T\times\widetilde{G}_q^r$.
The transition functions are of the form $(x,y,g)\mapsto(\psi(x,y),\gamma(y),\gamma_*(y)g)$.
Therefore, we have a foliation of $\widetilde{P}^r(\CF)$ locally defined by asking $y$ and $g$ to be constant, to which we refer as $\CF^r$.
We always equip $\widetilde{P}^r(\CF)$ with the foliation~$\CF^r$.

\begin{definition}
A connection on $\widetilde{P}^r(\CF)$ is said to be a \textit{Bott connection} if it is a Bott connection for $\CF^r$.
\end{definition}

The following is easy.

\begin{lemma}
A connection on $P^1(\CF)$ is a Bott connection if and only if it is associated with a Bott connection on $Q(\CF)$.
\end{lemma}

\begin{definition}[Canonical form]
Let $u\in\widetilde{P}^r(\CF)$ and $X\in T_u\widetilde{P}^r(\CF)$.
We choose a foliation chart $U\cong V\times T$ which contains $\pi^r(u)$.
Let $p\colon U\to T$ be the projection, $\theta_T$ the canonical form of $\widetilde{P}^r(T)$ and set
\[
\theta(X)=p^*\theta_T.
\]
We call $\theta$ the \textit{canonical form} on $\widetilde{P}^r(\CF)$.
\end{definition}
By Theorem~\ref{thm4.6}, the canonical form on $\widetilde{P}^r(\CF)$ is well-defined.

We have the following.
The proof is just a combination of Lemma~\ref{lem7.3} and Theorem~\ref{thmG11} so that omitted.

\begin{theorem}
There is a one-to-one correspondence between the following objects\/\textup{:}
\begin{enumerate}
\item
Bott connections on $Q(\CF)$.
\item
Sections of $\widetilde{P}^2(\CF)\to P^1(\CF)$ which are equivariant under the $\GL_q(\mathbb{R})$-actions and that preserve foliations.
\item
Sections of $\widetilde{P}^2(\CF)/\GL_q(\mathbb{R})\to M$ which preserve foliations.
\end{enumerate}
\end{theorem}

Next, we discuss deformations of foliations and Bott connections.

\begin{definition}
A connection $\nabla$ on $P^1(\CF)\ltimes\underline{\gl_q}(\mathbb{R})$ is said to be a \textit{Bott connection} if $\iota^*\nabla$ is a Bott connection on $P^1(\CF)$, where $\iota$ is the inclusion obtained by the natural identification of $P^1(\CF)$ with $P^1(\CF)\ltimes\{o\}$.
\end{definition}

Let $\omega=(\omega^i)$ be a local trivialization of $Q^*(\CF)$.
By the Frobenius theorem, there exists a $\gl_q(\mathbb{R})$-valued $1$-form, say $\theta$, such that
\[
d\omega+\theta\wedge\omega=0.
\]
The $1$-form $\theta$ is essentially the connection form of a Bott connection, say $\nabla$, with respect to the frame dual to $\omega$.
Assume now that $\CF$, $\nabla$, $\omega$, $\theta$ are smooth $1$-parameter families.
Let $t$ be the parameter and assume that $\CF_0=\CF$ and so on.
If we represent derivatives with respect to $t$ at $t=0$ by adding a dot, then we have
\[
d\dot\omega+\dot\theta\wedge\omega+\theta\wedge\dot\omega=0.
\]
By considering a foliation atlas, we consider $\omega$, etc.~are family defined on $M$.
Then a $\gl_q(\mathbb{R})$-valued $1$-form $\dot\theta$ gives rise to a global $\Hom(Q(\CF),Q(\CF))$-valued $1$-form on $M$ independent of the choice of foliation atlases.
In general, we adopt this property as an infinitesimal deformation of connections.

\begin{definition}[\cite{12}, see also~\cite{DuchampKalka}]
Let $\CF$ be a foliation and $\nabla$ a Bott connection on $Q(\CF)$.
We fix a family $\omega$ of local trivializations of $Q(\CF)$, and let $\theta$ be the family of the connection forms with respect to the dual of $\omega$.
A $Q(\CF)$-valued global $1$-form $\dot\omega$ and a $\Hom(Q(\CF),Q(\CF))$-valued $1$-form $\dot\theta$ are said to be \textit{infinitesimal deformations} of $\omega$ and $\theta$, respectively, if we have
\[
d\dot\omega+\theta\wedge\dot\omega+\dot\theta\wedge\omega=0
\]
in the sense that if we choose a local trivialization and if we represent $\omega,\theta,\dot{\omega}$ and $\dot\theta$ by components, then we have
\[
d\dot\omega^i+\theta^i{}_\alpha\wedge\dot{\omega}^\alpha+\dot{\theta}^i{}_\alpha\wedge\omega^\alpha=0.
\]
\end{definition}

By Lemma~\ref{lem7.3} and Theorem~\ref{thm6.16}, we have the following

\begin{theorem}
\label{thm7.12}
There is a one-to-one correspondence between the following\textup{:}
\begin{enumerate}
\item
Pairs of Bott connections on $Q(\CF)$ and their infinitesimal deformations.
\item
Bott connections on $P^1(\CF)\ltimes\underline{\gl_q}(\mathbb{R})$.
\item
Sections from $P^1(\CF)\ltimes\underline{\gl_q}(\mathbb{R})$ to $\widetilde{P}^2(\CF)\ltimes\underline{\widetilde{\mathfrak{g}}^2_q}$ equivariant under the $\GL_q(\mathbb{R})\ltimes\gl_q(\mathbb{R})$-action and preserving foliations.
\end{enumerate}
\end{theorem}

The normal bundle $Q(\CF)$ is also equipped with a foliation.
Indeed, if $(x,y,v)$ denote the coordinates for $Q(\CF)$ on a foliation chart, then we have $(\widehat{x},\widehat{y},\widehat{v})=(\psi(x,y),\gamma(y),D\gamma_y(v))$ so that foliations locally defined by $y=\text{const.}$ and $v=\text{const.}$ give rise to a foliation of $Q(\CF)$ which is represented by $\CF^{(2)}$.
It is easy to see that the normal bundle $Q(\CF^{(2)})$ is the $2$-normal bundle $Q^{(2)}(\CF)$ in~\cite{asuke:2015-2}.
The bundle $Q^{(2)}(\CF)$ plays a role of $T^2M$.
Namely, we always have the \textit{vertical subbundle} of $Q^{(2)}(\CF)$ which we represent as $Q(\CF)^V$.
On the other hand, given a Bott connection on $Q(\CF)$, we can define the \textit{horizontal subbundle} of $Q^{(2)}(\CF)$ which we represent by $Q(\CF)^H$.
If $\varpi$ denotes the projection from $Q(\CF)$ to $M$, then the both lifts are isomorphic to $\varpi^*(Q(\CF))$.
If in addition an infinitesimal deformation of the Bott connection is given, then we can define a connection on $Q^{(2)}(\CF)$ by considering $\begin{pmatrix}
\theta\\
\dot{\theta} & \theta
\end{pmatrix}$.
This is the construction given in~\cite{asuke:2015-2}, where characteristic classes for infinitesimal deformations of foliations are studied by means of these connections (cf.~\cite{Fuks}, \cite{13}).
They are obtained as differential forms on $Q(\CF)$ and then shown to project down to $M$.
If we make use of Theorem~\ref{thm7.12}, then we can obtain a Bott connection on $P^1(\CF)\ltimes\gl_q(\mathbb{R})$ which is valued in the Lie algebra of $\GL_q(\mathbb{R})\ltimes\gl_q(\mathbb{R})\cong T\GL_q(\mathbb{R})$, and we can avoid bundles over $Q(\CF)$ to obtain these classes.
If we denote by $\mathcal{B}$ the space of Bott connections on $P^1(\CF)\ltimes\underline{\gl_q}(\mathbb{R})$, then these classes are functionals on $\mathcal{B}$.
For example, we can consider the derivative of the Godbillon--Vey class, $D\mathrm{GV}$ for short, with respect to infinitesimal deformations of foliations.
It is known that if the foliation under consideration admits a transverse projective structure (not necessarily flat), then $D\mathrm{GV}$ vanishes for any infinitesimal deformations of foliations~\cite{asuke:2015}.
This means that if $\CF$ admits a transverse projective structure, then $D\mathrm{GV}$ as a functional on $\mathcal{B}$ is identically equal to zero.
Similarly, some characteristic classes are introduced for deformations of flat connections in~\cite{Lue}.
If $\mathcal{C}$ denotes the space of connections on $P^1(M)\ltimes\underline{\gl_n}(\mathbb{R})$, then these classes can be regarded as functionals on $\mathcal{C}$ and results can be understood as properties of such functionals.

\end{document}